\pdfoutput=1
\documentclass[journal]{IEEEtran}
\usepackage{cite}
\usepackage{amsmath,amssymb,amsfonts}
\usepackage{amsthm}
\usepackage{nicefrac}
\usepackage{mathtools}
\usepackage{dutchcal}
\usepackage[protrusion=true,expansion=true]{microtype} 
\usepackage{graphicx} 
\usepackage{subfigure}
\usepackage{epstopdf}
\usepackage{wrapfig} 
\usepackage{float}
\usepackage{mathrsfs}
\usepackage{amsmath,bm}
\usepackage{color,soul}
\usepackage{enumerate}
\usepackage{multirow}
\usepackage[export]{adjustbox}

\newcommand{\etal}{\textit{et al.~}}
\newcommand{\ie}{\textit{i}.\textit{e}.,~}

\DeclareMathOperator*{\argmin}{arg\,min}

\newtheorem{thm}{Theorem}
\newtheorem{lemma}[thm]{Lemma}
\newtheorem{prop}[thm]{Proposition}

\usepackage{algorithmicx,algorithm}
\usepackage{array}
\usepackage{multirow}
\usepackage{listliketab}

\begin{document}
	
	\title{A Fast Convergent Ordered-Subsets Algorithm with Subiteration-Dependent Preconditioners for PET Image Reconstruction}
	
	\author{Jianfeng~Guo,~C.~Ross~Schmidtlein,~Andrzej~Krol,~Si~Li,~Yizun~Lin,~Sangtae~Ahn, Charles~Stearns,~and~Yuesheng~Xu
		\thanks{J. Guo is supported by the Special Project on High-performance Computing under the National Key R\&D Program (No. 2016YFB0200602) and by the National Natural Science Foundation of China under grant 11771464; C. R. Schmidtlein is partially supported by the Memorial Sloan Kettering Cancer Center's support grant from the National Cancer Institute (P30 CA008748); C. R. Schmidtlein, A. Krol and Y. Xu are  partially supported by the National Cancer Institute of the National Institutes of Health under Award Number R21CA263876; S. Li is supported by the Natural Science Foundation of Guangdong Province under grant 2022A1515012379, by the Opening Project of Guangdong Province Key Laboratory of Computational Science at Sun Yat-sen University under grant 2021007 and by the National Natural Science Foundation of China under grant 11771464; Y. Lin is supported by Guangdong Basic and Applied Basic Research Foundation under grant 2021A1515110541, by the Fundamental Research Funds for the Central Universities of China under grant 21620352 and by the Opening Project of Guangdong Province Key Laboratory of Computational Science at Sun Yat-sen University under grant 2021006; Y. Xu is supported by the US National Science Foundation under grant DMS-1912958. The content is solely the responsibility of the authors and does not necessarily represent the official views of the National Institutes of Health.  \emph{(Corresponding author: Yuesheng Xu.)}}
		\thanks{Jianfeng Guo is with the School of Computer Science and Engineering, Sun Yat-sen University, Guangzhou, 510275, China (guojf3@mail2.sysu.edu.cn).}
		\thanks{C. Ross Schmidtlein is with the Department of Medical Physics, Memorial Sloan Kettering Cancer Center, New York, NY 10065, USA (schmidtr@mskcc.org).}
		\thanks{Andrzej Krol is with the Departments of Radiology and Pharmacology, SUNY Upstate Medical University, 750 East Adams Street, Syracuse
			NY 13210, USA (krola@upstate.edu).}
		\thanks{Si Li is with School of Computer Science and Technology, Guangdong University of Technology, Guangzhou 510006, China (sili@gdut.edu.cn).}
		\thanks{Yizun Lin is with Department of Mathematics, College of Information Science and Technology, Jinan University, Guangzhou 510632, China (linyizun@jnu.edu.cn).}
		\thanks{Sangtae Ahn is with GE Research, Niskayuna, NY 12309, USA (ahns@ge.com).}
		\thanks{Charles Stearns is with GE Healthcare, Waukesha, WI 53188, USA (stearnsc@alum.mit.edu).}
		\thanks{Yuesheng Xu is with Department of Mathematics and Statistics, Old Dominion University, Norfolk, VA 23529, USA (y1xu@odu.edu).}
	}
	
	\def \cA {{\cal A}}
	\def \cB {{\cal B}}
	\def \cD {{\cal D}}
	\def \cE {{\cal E}}
	\def \cF {{\cal F}}
	\def \cG {{\cal G}}
	\def \cH {{\cal H}}
	\def \cI {{\cal I}}
	\def \cJ {{\cal J}}
	\def \cK {{\cal K}}
	\def \cL {{\cal L}}
	\def \cM {{\cal M}}
	\def \cN {{\cal N}}
	\def \cO {{\cal O}}
	\def \cP {{\cal P}}
	\def \cR {{\cal R}}
	\def \cS {{\cal S}}
	\def \cU {{\cal U}}
	\def \cV {{\cal V}}
	\def \cW {{\cal W}}
	\def \cX {{\cal X}}
	\def \cY {{\cal Y}}
	\def \cZ {{\cal Z}}
	
	\def \mcI {\mathcal{I}}
	\def \mcN {\mathcal{N}}
	\def \mcT {\mathcal{T}}
	
	\renewcommand{\Bbb}{\mathbb}
	\def \bR {\Bbb R}
	\def \bN {\Bbb N}
	\def \bS {\Bbb S}
	
	\def \bfA {{\bm A}}
	\def \bfB {{\bm B}}
	\def \bfD {{\bm D}}
	\def \bfE {{\bm E}}
	\def \bfG {{\bm G}}
	\def \bfH {{\bm H}}
	\def \bfI {{\bm I}}
	\def \bfP {{\bm P}}
	\def \bfQ {{\bm Q}}
	\def \bfR {{\bm R}}
	\def \bfS {{\bm S}}
	\def \bfT {{\bm T}}
	\def \bfW {{\bm W}}
	\def \bfLambda {{\bm\Lambda}}
	
	\def \bfgamma {{\bm\gamma}}
	\def \bfepsilon {{\bm\epsilon}}
	\def \bfalpha {{\bm\alpha}}
	\def \bfbeta {{\bm\beta}}
	\def \bflambda {{\bm\lambda}}
	\def \bfdelta {{\bm\delta}}
	\def \bftheta {{\bm\theta}}
	\def \bfnu {{\bm\nu}}
	\def \bfmu {{\bm\mu}}
	\def \bfb {{\bm b}}
	\def \bfc {{\bm c}}
	\def \bfd {{\bm d}}
	\def \bfe {{\bm e}}
	\def \bff {{\bm f}}
	\def \bfg {{\bm g}}
	\def \bfh {{\bm h}}
	\def \bfJ {{\bm J}}
	\def \bfq {{\bm q}}
	\def \bfu {{\bm u}}
	\def \bfv {{\bm v}}
	\def \bfw {{\bm w}}
	\def \bfx {{\bm x}}
	\def \bfy {{\bm y}}
	\def \bfz {{\bm z}}
	
	\def \bRM {\bR^{M}}
	\def \bRm {\bR^{m}}
	\def \bRmo {\bR^{m_0}}
	\def \bRn {\bR^{n}}
	\def \bRd {\bR^{d}}
	\def \bRpd {\bR_+^{d}}
	\def \bRpn {\bR_+^{n}}
	\def \bRnn {\bR^{n\times n}}
	\def \bRmm {\bR^{m\times m}}
	\def \bRmomo {\bR^{m_0\times m_0}}
	\def \bRmod {\bR^{m_0\times d}}
	\def \bRdd {\bR^{d\times d}}
	\def \bSn {\bS^{n}}
	\def \bSpd {\bS_+^{d}}
	\def \bSpn {\bS_+^{n}}
	\def \bSpm {\bS_+^{m}}
	\def \bSpmo {\bS_+^{m_0}}
	\def \indi {\iota_{\bR_+^d}}
	\def \prox {\mathrm{prox}}
	\def \diag {\mathrm{diag}}
	\def \TMC {\text{TMC}}
	\def \EM {\text{EM}}
	\def \IEM {\text{IEM}}
	
	\def \xk {{\bm x}^{k}}
	\def \yk {{\bm y}^{k}}
	\def \zk {{\bm z}^{k}}
	\def \fk {{\bm f}^{k}}
	\def \hk {{\bm h}^{k}}
	\def \bk {{\bm b}^{k}}
	\def \ck {{\bm c}^{k}}
	\def \pk {p^{k}}
	\def \qk {{\bm q}^{k}}
	\def \tk {t^{k}}
	\def \uk {{\bm u}^{k}}
	\def \vk {{\bm v}^{k}}
	
	\def \xkk {{\bm x}^{k+1}}
	\def \ykk {{\bm y}^{k+1}}
	\def \zkk {{\bm z}^{k+1}}
	\def \fkk {{\bm f}^{k+1}}
	\def \hkk {{\bm h}^{k+1}}
	\def \bkk {{\bm b}^{k+1}}
	\def \ckk {{\bm c}^{k+1}}
	\def \pkk {p^{k+1}}
	\def \qkk {{\bm q}^{k+1}}
	\def \tkk {t^{k+1}}
	\def \ukk {{\bm u}^{k+1}}
	\def \vkk {{\bm v}^{k+1}}
	
	\def \tu {\tilde{{\bm u}}}
	\def \tuk {\tilde{{\bm u}}^{k}}
	\def \tvk {\tilde{{\bm v}}^{k}}
	\def \txk {\tilde{{\bm x}}^{k}}
	\def \tyk {\tilde{{\bm y}}^{k}}
	\def \tzk {\tilde{{\bm z}}^{k}}
	\def \tfk {\tilde{{\bm f}}^{k}}
	\def \thk {\tilde{{\bm h}}^{k}}
	\def \tbk {\tilde{{\bm b}}^{k}}
	\def \tck {\tilde{{\bm c}}^{k}}
	\def \tqk {\tilde{{\bm q}}^{k}}
	
	\def \tbkk {\tilde{{\bm b}}^{k+1}}
	\def \tckk {\tilde{{\bm c}}^{k+1}}
	\def \tqkk {\tilde{{\bm q}}^{k+1}}
	\def \tukk {\tilde{{\bm u}}^{k+1}}
	\def \tvkk {\tilde{{\bm v}}^{k+1}}
	\def \txkk {\tilde{{\bm x}}^{k+1}}
	\def \tykk {\tilde{{\bm y}}^{k+1}}
	\def \tfkk {\tilde{{\bm f}}^{k+1}}
	\def \thkk {\tilde{{\bm h}}^{k+1}}
	
	\def \fkl {{\bm f}^{k,l}}
	\def \fkll {{\bm f}^{k,l-1}}
	\def \fkij {\bff^{k,i}_j}
	\def \fkiij {\bff^{k,i-1}_j}
	\def \tfki {\tilde{{\bm f}}^{k,i}}
	\def \tfkii {\tilde{{\bm f}}^{k,i-1}}
	\def \hkl {{\bm h}^{k,l}}
	\def \hkll {{\bm h}^{k,l-1}}
	\def \bkl {{\bm b}^{k,l}}
	\def \bkll {{\bm b}^{k,l-1}}
	\def \tbkl {\tilde{{\bm b}}^{k,l}}
	\def \tbkll {\tilde{{\bm b}}^{k,l-1}}
	\def \ckl {{\bm c}^{k,l}}
	\def \ckll {{\bm c}^{k,l-1}}
	\def \tckl {\tilde{{\bm c}}^{k,l}}
	\def \tckll {\tilde{{\bm c}}^{k,l-1}}
	\def \Skl {{\bm S}^{k,l}}
	\def \Skll {{\bm S}^{k,l-1}}
	
	\maketitle

	\begin{abstract}
		We investigated the imaging performance of a fast convergent ordered-subsets algorithm with subiteration-dependent preconditioners (SDPs) for positron emission tomography (PET) image reconstruction. In particular, we considered the use of SDP with the block sequential regularized expectation maximization (BSREM) approach with the relative difference prior (RDP) regularizer due to its prior clinical adaptation by vendors. Because the RDP regularization promotes smoothness in the reconstructed image, the directions of the gradients in smooth areas more accurately point toward the objective function’s minimizer than those in variable areas. Motivated by this observation, two SDPs have been designed to increase iteration step-sizes in the smooth areas and reduce iteration step-sizes in the variable areas relative to a conventional expectation maximization preconditioner.  The momentum technique used for convergence acceleration can be viewed as a special case of SDP.
		We have proved the global convergence of SDP-BSREM algorithms by assuming certain characteristics of the preconditioner. 
		By means of numerical experiments using both simulated and clinical PET data, we have shown that the SDP-BSREM algorithms substantially improve the convergence rate, as compared to conventional BSREM and a vendor's implementation as Q.Clear. Specifically, SDP-BSREM algorithms converge 35\%-50\% faster in reaching the same objective function value than conventional BSREM and commercial Q.Clear algorithms. Moreover, we showed in phantoms with hot, cold and background regions that the SDP-BSREM algorithms approached the values of a highly converged reference image faster than conventional BSREM and commercial Q.Clear algorithms.
	\end{abstract}
	
	
	\section{Introduction}
	
	\IEEEPARstart{P}{ositron} emission tomography (PET) data are inherently count limited due to health consideration, basic physical processes, and patient tolerance.  Moreover, these data must be reconstructed into images within a few minutes of acquisition.  This creates a challenging situation in which vendors strive to produce high quality images in a clinically viable time frame. In this study, we introduce a method for accelerating the reconstruction of high quality PET images.
	
	Over last 20 years, the non-penalized maximum-likelihood (ML) statistical approaches have become a preferred model for the reconstruction of PET \cite{Shepp1982maximum,Lange1984em}. However, when iterated to full convergence, ML methods produce extremely noisy images, and are sensitive to small statistical perturbations in the data. Hence, these methods are seldom run to full convergence and iterations are stopped before fitting noise becomes unacceptable at the expense of excessive blur in the reconstructed images.  It has been demonstrated that applications of penalized likelihood (PL) models that include a data fidelity term (Kullback-Leibler divergence) and a regularization term leads to improved quantification and better noise suppression, as compared to non-penalized reconstructions \cite{ahn2015quantitative}. To reduce the computational expense, ordered-subsets expectation maximization (OSEM) algorithms proposed by Hudson and Larkin are widely used in unregularized PET image reconstruction \cite{Hudson1994Accelerated}.  However, OSEM is unsuitable for regularized image reconstruction leading to the development of relaxation \cite{Browne1996alternative,Bertsekas1997A} and the block sequential regularized expectation maximization (BSREM) algorithm \cite{de2001fast}.
	
	Initially, quadratic penalties were explored \cite{ahn2003globally,tsai2018fast}, but they had resulted in over-smoothed edges and loss of details in the reconstructed images. Later, a number of other penalties were developed to address these problems but they often had undesirable properties including nonsmooth \cite{rudin1992nonlinear,bredies2010total}, noncovex \cite{Geman1987Statistical}, or requiring additional hyper-parameters \cite{Mumcuoglu1996Bayesian}. 
	An example of such an approach is the total variation penalty that is able to preserve sharp boundaries between low-variability regions \cite{rudin1992nonlinear,Zhang2016investigation}. Thus, ability to deal with non-smooth priors became an urgent issue, However, only a few reconstruction algorithms have been able to combine the Poisson noise model and non-smooth priors \cite{chambolle2011first,krol2012preconditioned,li2015effective}.\par
	
	In an alternative approach, Nuyts \etal \cite{nuyts2002concave} introduced the relative difference prior (RDP) that preserved high spatial frequencies in reconstructed images while still being smooth and convex. This RDP was adopted by General Electric (GE) Healthcare as the penalty term in their PL PET reconstruction model. The penalty term is controlled by a single user-defined parameter called beta. The GE Healthcare introduced a modified BSREM algorithm \cite{ahn2003globally} to solve the PL model in their commercial clinical software, called Q.Clear, that is currently available on GE PET/CT scanners \cite{ahn2015quantitative}. Other interesting methods, suitable for optimization with smooth penalties, include the optimization transfer descent algorithm (OTDA) \cite{wang2012penalized,wang2015edge} and the preconditioned limited-memory Broyden-Fletcher-Goldfarb-Shanno with boundary constraints algorithm (L-BFGS-B-PC) \cite{tsai2018fast}.  While these algorithms converge very rapidly, they represent a substantial departure from the BSREM algorithm complicating their implementation.
	
	The choice of preconditioners in the algorithm is well known to strongly affect the convergence rate \cite{mumcuoglu1994fast,krol2012preconditioned,tsai2018fast,lin2019krasnoselskii}. The widely used preconditioners have been designed based on the EM matrix \cite{mumcuoglu1994fast,krol2012preconditioned,lin2019krasnoselskii,Ehrhardt2019fast} or the Hessian matrix \cite{tsai2018fast}. As part of the BSREM convergence proof, Ahn and Fessler \cite{ahn2003globally} presented the subiteration-independent preconditioner, which can be viewed as a uniform operator of the image for all subiteration. However, a subiteration-independent preconditioner is overly restrictive and may result in a slower convergence rate. We believe that a well-designed subiteration-dependent preconditioner (SDP) 
	will accelerate the algorithm convergence.
	
	In the present study, we propose a subiteration-dependent preconditioned BSREM (SDP-BSREM) for the RDP regularized PET image reconstruction. We prove that it is convergent under certain assumptions imposed on the preconditioner. According to the smoothness-promoting property of the RDP regularization in the reconstructed image, the directions of the gradients in smooth areas more accurately point toward the objective function's minimizer than those in variable areas. Inspired by this observation, we propose two SDPs satisfying the assumptions needed for the convergence proof. We note that the momentum technique is a special case of SDP. We have used the numerical gradient of the image to measure its smoothness. These two SDPs achieve larger step-sizes in the smooth areas of the image and smaller step-sizes in the variable areas of the image. The proposed algorithms have been compared with BSREM for simulations and with the Q.Clear method with data acquired from a GE PET/CT. In simulations, two numerical phantoms were used. In the clinical comparisons, data from a whole-body PET patient and an American College of Radiology (ACR) quality assurance phantom (Esser Flangeless PET phantom) were used both with and without time-of-flight data. 
	
	This paper is organized in five sections. In section II, we first describe the RDP regularized PET image reconstruction model and the modified BSREM algorithm and then develop our new SDP-BSREM algorithm. In section III, proofs for convergence of SDP-BSREM are provided with and without an interior assumption and four SDPs satisfying the convergence conditions are presented as well. Comprehensive comparisons of the results obtained for simulated and clinical data obtained by means of our proposed SDP-BSREM methods versus BSREM and Q.Clear are provided in section IV. The conclusions are presented in section V. Two appendices with additional details are also provided.
	
	\section{Methodology}
	
	In this section, we develop the SDP-BSREM algorithm for solving the RDP regularized PET image reconstruction model.
	
	\subsection{RDP Regularized PET Image Reconstruction Model}
	We denote by $\bR_+$ the set of all nonnegative real numbers, by $ \bR_{++}$ the set of all positive real numbers, by $\bN$ the set of positive integers, and by $ \bN_0:=\bN\cup\left\lbrace0\right\rbrace $.
	For $p, q\in \bN$, we let $\bfA\in\bR_+^{p\times q}$ denote the PET system matrix whose entries are the probability of detection of the positron annihilation gamma photon pairs emitted from a particular voxel containing PET radiotracer, and let $\bfgamma\in\mathbb{R}_{++}^{p}$ denote the mean value of the background events produced by random and scatter coincidences. 
	The relation of the radiotracer distribution $\bff\in\bR^q_+$ within a patient with the projection data $\bfg\in\bR_{+}^{p}$ acquired by a PET scanner 
	is described by the Poisson model
	\begin{equation}\label{model-Poisson}
	\bfg=\text{Poisson}(\bfA\bff+\bfgamma),
	\end{equation}
	where $\text{Poisson}(\bfx)$ denotes a Poisson-distributed random vector with mean $\bfx$.
	
	Model \eqref{model-Poisson} may be solved by minimizing the fidelity term
	\begin{equation}\label{eq:FidelityTerm}
	F(\bff) \coloneqq  \langle\bfA\bff,{\bf1}_p\rangle-\langle\ln\left(\bfA\bff+\bfgamma\right),\bfg\rangle,
	\end{equation}
	where ${\bf1}_p\in\bR^p$ denotes the vector with all components 1, $\ln\bfx\coloneqq [\ln x_1, \ln x_2,\dotsc,\ln x_n]^\top$ is the logarithmic function at a vector $\bfx\in\bRn_{++}$ and $ x_i $ is the $ i $-th component of $ \bfx $, and $\langle \bfx,\bfy\rangle\coloneqq \sum_{i=1}^{n}x_iy_i$ denotes the inner product of $\bfx,\bfy\in\bRn$. It is well-known that model \eqref{eq:FidelityTerm} is ill-posed \cite{yu2002edge} and its solutions may result in over-fitting in reconstructed images. Regularization is often used to avoid the over-fitting problem. A commonly used regularized PET image reconstruction model has the following form: 
	\begin{equation}\label{model-RDP}
	\argmin_{\bff\in\bR^q_+} \Phi(\bff),
	\end{equation}
	where
	\begin{equation}\label{eq:ObjectiveFunction}
	\Phi(\bff)\coloneqq  F(\bff)+\beta R(\bff),
	\end{equation}
    with $ \beta\in\bR_+ $ being the regularization parameter and $R(\bff)$ representing the regularization term. In this study, we will consider the RDP \cite{nuyts2002concave} regularization term that is given by
	\begin{equation}
	R(\bff) \coloneqq  \sum_{j=1}^q{\sum_{k \in N_j}{ \frac{(f_j-f_k)^2}{(f_j+f_k)+\gamma_{R}\left|f_j-f_k\right|+\epsilon}}},
	\label{eq:RDP}
	\end{equation}
	where $\gamma_{R}\in\mathbb{R}_{+}$ controls the degree of edge preservation, and $N_j$ is the neighborhood of pixel $j$. 
	
	In model \eqref{eq:RDP}, a small constant $\epsilon>0 $ is added to the denominator to avoid singularities when both $ f_j $ and $ f_k $ are equal to zero. By its definition, RDP is a function of both differences of neighboring pixels and their sums. The inclusion of the sums term makes RDP differ from conventional regularization terms and causes the regularizer to be activity-dependent. We note that the function $ \Phi(\bff) $ is twice differentiable since both $ F(\bff) $ and $ R(\bff) $ are twice differentiable \cite{krol2012preconditioned,nuyts2002concave}. The inclusion of a small constant $\epsilon$ in the denominator of RDP provides the objective function $\Phi$ with two useful properties (the proofs are provided in appendix \ref{appendix:ConvexAndGradient}): (i) it is strictly convex under an assumption that $\bfA^\top \bfg$ is a nonzero vector; and (ii) it has a Lipschitz continuous gradient on $\bR^q_+$. 
	
	\subsection{Modified BSREM Algorithm}
	A modified BSREM \cite{ahn2003globally} was adopted by GE Healthcare as the optimizer in the Q.Clear method \cite{ahn2015quantitative} for solving the model \eqref{model-RDP}. Here we describe and review the modified BSREM.\par
	
	Minimization problem  \eqref{model-RDP} is often solved by the gradient descent method.
	However, computing the whole gradient $ \nabla\Phi $ is computationally expensive. To alleviate this issue, the ordered-subsets (OS) algorithm was developed to accelerate its convergence \cite{Hudson1994Accelerated,Browne1996alternative,Bertsekas1997A,de2001fast}. For  $n\in \bN$, let $\bN_n\coloneqq \left\lbrace 1,2,\ldots,n\right\rbrace.$  For $M\in \bN$, let ${\cal I}:=\{I_i: i\in\bN_M\}$ be a collection of disjoint subsets of $\bN_p$ such that $\bigcup_{i=1}^M I_i=\bN_p.$ The partition $\cI$ is chosen as in \cite{ahn2003globally}. According to the partition ${\cal I}$, we partition the system matrix $ \bfA $ into $M$ row sub-matrices $\bfA_i$, and $\bfg$ and $\bfgamma$ into $M$ sub-vectors $\bfg_i$ and  $\bfgamma_i$, respectively, for $i\in\bN_M$. We use $|\Omega|$ to denote the cardinality of set $\Omega$.
	For $i\in\bN_M$, we define
	\begin{equation}\label{eq:SubObjFunc}
	\Phi_i(\bff)\coloneqq \langle\bfA_i\bff,{\bf1}_{\left|I_i\right|}\rangle-\langle\ln\left(\bfA_i\bff+\bfgamma_i\right),\bfg_i\rangle+\frac{\beta}{M} R(\bff).
	\end{equation}
	It follows that
	$\Phi(\bff)=\sum_{i=1}^M \Phi_i(\bff).$ An OS algorithm computes only one $\nabla\Phi_i$ at each subiteration step.
	
	A subiteration-independent preconditioned OS algorithm was proposed in \cite{ahn2003globally}. The  preconditioner is designed by using an upper bound of the solution set of minimization problem  \eqref{model-RDP}.
	It was proved in \cite{ahn2003globally} that for any projection data $\bfg,$ there exists a constant $U>0$ such that the solution set $\cS^*$ of minimization problem \eqref{model-RDP} is contained in the bounded set
	\begin{equation}\label{eq:BoundedSet}
	\cB\coloneqq \left\lbrace \bff:\bff\in\bR^q_+, \ 0\leqslant f_j\leqslant U,\ j\in\bN_q\right\rbrace.
	\end{equation}
	That is, $\cS^*\subset \cB$.
	For $ \bff\in\cB, $ a subiteration-independent diagonal preconditioner $ \bfS(\bff) $ is defined as 
	\begin{equation}\label{eq:preconditioner}
	S(\bff)_{jj}\coloneqq \begin{cases}
	f_j/p_j & \text{if } 0\leq f_j < U/2,\\
	(U-f_j)/(p_j) & \text{if } U/2\leq f_j \le U,
	\end{cases}
	\end{equation}
	where $p_j$ are defined by
	\begin{equation}\label{eq:pj}
	{p_j}\coloneqq \begin{cases}
	    (\bfA^{\top}{\bf1}_p)_j/M & \text{if } (\bfA^\top{\bf1}_p)_j>0,\\
	    1/M & \text{if } (\bfA^\top{\bf1}_p)_j=0,
	    \end{cases}\ \ \mbox{for}\ \ j\in\bN_q.
	\end{equation}
	Note that the preconditioner $S$ is uniform for all iterations.

	For a small $t\in(0,U)$ and $ \bff\in\bR^q, $ an operator $ P_{t} : \bR^q \to \cB $ is defined by
	\begin{equation}\label{eq:OperatorPt}
	P_t(\bff)_j\coloneqq \begin{cases}
	t & \text{if } f_j \leqslant 0,\\
	U-t & \text{if } f_j\geqslant U,\\
	f_j & \text{otherwise}.
	\end{cases} 
	\end{equation}
Using operator $P_{t}$, the modified BSREM algorithm \cite{ahn2003globally} may be described as for $k\in\bN_0$, $i\in\bN_M$,
	\begin{equation}
	\begin{dcases}
	\tilde{\bff}^{k,i}=\bff^{k,i-1}-\lambda_k\bfS(\bff^{k,i-1})\nabla\Phi_i(\bff^{k,i-1}),\\
	\bff^{k,i}=P_{t}(\tilde{\bff}^{k,i}),\\
	\end{dcases} 
	\end{equation}
	with $ \bff^{k,0}:=\bff^k$, $\bff^{k+1}:=\bff^{k,M},$
	where $\lambda_k>0$ is the relaxation parameter. For simplicity of notation, we will refer to the modified form of BSREM simply as BSREM.
	
	\subsection{BSREM with Subiteration-Dependent Preconditioners}
	In this subsection we propose subiteration-dependent preconditioners (SDPs). To motivate them, we review the momentum approach. The momentum is an acceleration technique widely used in optimization \cite{Nesterov1983Method, beck2009fast, Zeng2018FixFISTA}. The Nesterov momentum \cite{Nesterov1983Method} has been combined with OS by Kim \etal \cite{kim2015combining} for CT image reconstruction. Recently, Lin \etal \cite{lin2019krasnoselskii} successfully applied a different form of momentum to PET image reconstruction. However, no explicit convergence proof has been provided for the OS combined momentum methods. Instead, Kim \etal proved that the expectation of the successive steps converged, while Lin \etal proved convergence for the non-OS method. The momentum technique used in \cite{lin2019krasnoselskii} can be described as follows:  for $k\in\bN_0$, $i\in\bN_M$,
	\begin{equation}
	\begin{dcases}
	\tilde{\bff}^{k,i}=\max[\bff^{k,i-1}-\lambda_k\bfS(\bff^{k,i-1})\nabla\Phi_i(\bff^{k,i-1}),0],\\
	\bff^{k,i}=(1-\alpha_{k,i})\bff^{k,i-1}+\alpha_{k,i}\tilde{\bff}^{k,i},
	\end{dcases}
	\end{equation}
	where $\alpha_{k,i}>1$ is the momentum sequence. Under an assumption that  $\tilde{\bff}^{k,i}$ are non-negative, one can obtain 
	\begin{equation}\label{eq:KMmomentum}
	\bff^{k,i}=\bff^{k,i-1}-\lambda_k\alpha_{k,i}\bfS(\bff^{k,i-1})\nabla\Phi_i(\bff^{k,i-1}).
	\end{equation}
	By letting $\bfS^{k,i}(\bff)\coloneqq \alpha_{k,i}\bfS(\bff)$, we may reinterpret $\bfS^{k,i}$ as an  SDP and \eqref{eq:KMmomentum} as a BSREM algorithm with  $\bfS^{k,i}$. Inspired by this, we introduce SDPs by setting $\bfS^{k,i}(\bff)\coloneqq \diag(\bfalpha^{k,i})\bfS(\bff)$, where $\bfalpha^{k,i}$ is a positive vector sequence and $\diag ({\bfy})$ denotes the diagonal matrix with the diagonal entries being the components of the vector ${\bfy}$. Using the same notation as for BSREM above, we have arrived at the SDP-BSREM algorithm for solving model \eqref{model-RDP} given in Table \ref{tab:SDPBSREM}.
	
	\begin{table}[htp!]
		\centering
		\caption{SDP-BSREM Algorithm}
		\begin{tabular}{|l|}
			\hline
			\begin{math}
			\begin{aligned}
			&\text{Preparation:~} \bff^0, M, T, P_{t} \text{~is~defined~in~\eqref{eq:OperatorPt}}\\
			&\textbf{for}~ k=0,1,2,\dotsc,T\\
			&\quad\bff^{k,0}=\fk\\
			&\quad\textbf{for}~ i=1,2,\dotsc,M\\
			&\quad\quad\tilde{\bff}^{k,i}=\bff^{k,i-1}-\lambda_k\bfS^{k,i}(\bff^{k,i-1})\nabla\Phi_i(\bff^{k,i-1})\qquad \\
			&\quad\quad\bff^{k,i}=P_{t}(\tilde{\bff}^{k,i})\\
			&\quad\textbf{end}\\
			&\quad\bff^{k+1}=\bff^{k,M}\\
			&\textbf{end}\\
			\end{aligned}\end{math}\\
			\hline
		\end{tabular}
		\label{tab:SDPBSREM}
	\end{table}
	
	Bearing in mind the momentum concept defined in \cite{lin2019krasnoselskii}, it is clear that the iteration sequence provided in \eqref{eq:KMmomentum} is a special case of our proposed SDP-BSREM algorithm with $\bfalpha^{k,i}:=\alpha_{k,i}{\bf1}_q$. For this reason, we expect that our proposed SDP-BSREM algorithm setting will allow us to choose an SDP to yield the convergence acceleration. We will evaluate its performance by means of numerical experiments to be presented in section \ref{sec:NumRes}.\par

	\section{Convergence of SDP-BSREM algorithm}
	\label{section_Convergence_of_SDP-BSREM_algorithm}
	
	In this section we present convergence properties of the SDP-BSREM algorithm. We also describe four specific SDPs that satisfy the convergence condition. To this end, we assume that the objective function $ \Phi $ satisfies the hypothesis: 
	\begin{enumerate}[(i).]
		\item $\Phi$ has a unique minimizer on $\cB$; 
		\item $ \Phi$ is convex and twice differentiable on $ \cB $; 
		\item $\nabla\Phi_i$ are Lipschitz continuous on $\cB$ for all $i\in\bN_M$.
	\end{enumerate}
	\par
	The use of SDPs results in scaled subset gradients with their sum inconsistent with the scaled full gradient. It complicates the convergence proof and requires additional assumptions on the preconditioner to make the inconsistencies asymptotically approach zero. For a general SDP, our proposed SDP-BSREM algorithm may not converge \cite{ahn2003globally}. Nevertheless, by imposing certain assumptions on $ \bfS^{k,i}, $ convergence to the desired optimum point can be ensured. For a SDP $ \bfS^{k,i}(\bff) = \diag(\bfalpha^{k,i})\bfS(\bff) $ with $\bfS(\bff)$ defined in \eqref{eq:preconditioner}, the required additional assumptions are as follows: 
	\begin{enumerate}[(i).]\addtocounter{enumi}{+3}
		\item The relaxation sequence satisfies $\sum_{k=0}^\infty\lambda_k=\infty$ and $\sum_{k=0}^\infty\lambda^2_k<\infty$;
		\item There exists a positive vector $\bfalpha$ such that $\lim_{k\to\infty}\bfalpha^{k,i} =\bfalpha $ for all $ i\in\bN_M $;
		\item The vector series $ \sum_{k=0}^\infty\lambda_k(\bfalpha-\bfalpha^{k,i})$ converge for all $i\in\bN_M$. 
	\end{enumerate}
	\par
	Note that condition (iv) was imposed in \cite{ahn2003globally,Li2006A} for convergence proofs for relaxed OS algorithms. Conditions (v) and (vi) were imposed to overcome difficulties caused by the use of different preconditioners in different subiterations.
	
	We now state a lemma regarding the Lipschitz continuity properties of SDPs. The proof is included in Appendix \ref{appendix:lemconvergence}.
	\begin{lemma}\label{lem:LipConPrecond}
		If conditions (iii) and (v) are satisfied, then $ \bfS^{k,i}(\bff)\nabla\Phi_i(\bff) $ are uniformly bounded and  Lipschitz continuous on $ \cB $ with Lipschitz constants bounded above by a uniform constant, for all $ k\in\bN_0, i\in\bN_M$.
	\end{lemma}
	
	The inclusion of the operator $P_{t}$ in the SDP-BSREM algorithm complicates the convergence proof. Here, we present our approach in dealing with this difficulty. Let $\mathrm{int}\,\cB$ denote the interior of $\cB$. We prove the convergence of SDP-BSREM in two steps. We first prove it with the interior assumption $\tilde{\bff}^{k,i}\in\mathrm{int\,}\cB$ for all $k\in\bN_0, i\in\bN_M$, and then prove it by showing that the interior assumption holds true under certain conditions.\par
	We now proceed the first step. If $\tilde{\bff}^{k,i}\in\mathrm{int\,}\cB$ for all $k\in\bN_0, i\in\bN_M$, then $\bff^{k,i}=P_{t}(\tilde{\bff}^{k,i})=\tilde{\bff}^{k,i}$ and the iteration scheme can be formulated as
	\begin{equation}\label{eq:IterationInterior}
		\bff^{k,i}=\bff^k-\lambda_k\sum_{l=1}^i\bfS^{k,l}(\bff^{k,l-1})\nabla\Phi_t(\bff^{k,l-1}).
	\end{equation}
	
	We first establish a technical lemma. 
	\begin{lemma}\label{lem:difference}
	Suppose conditions (iii) and (v) are satisfied. If $\lim_{k\to\infty}\lambda_k=0,$ and $\tilde{\bff}^{k,i}\in\mathrm{int\,}\cB$, for all $k\in\bN_0$, $i\in\bN_M,$ then
			$\lim_{k\to\infty}(\bff^{k,i}-\bff^k)=0$, for all $i\in\bN_M.$
	\end{lemma}
	\begin{proof}
		By Lemma \ref{lem:LipConPrecond}, $\bfS^{k,i}(\bff)\nabla\Phi_i(\bff)$, $k\in\bN_0,i\in\bN_M$, are uniformly bounded on $\cB$. This combined with $\lim_{k\to\infty}\lambda_k=0$ and \eqref{eq:IterationInterior} yields the desired result.
	\end{proof}
	
	Let $ \bfdelta^{k,i}\coloneqq \bfalpha-\bfalpha^{k,i} , \delta_k\coloneqq\max_{i\in\bN_M,j\in\bN_q}|\delta^{k,i}_j|. $ We state a technical lemma whose proof is included in Appendix \ref{appendix:lemconvergence}.
	\begin{lemma}\label{lem:monotone}
		Suppose conditions (iii)-(vi) are satisfied and $\tilde{\bff}^{k,i}\in\mathrm{int\,}\cB$ for all $ k\in\bN_0,  i\in\bN_M$. If $ f^{k,i-1}_j\in(0,U/2) $ for all $ i\in\bN_M $, then
		\begin{equation}\label{eq:increasing}
			f^{k}_{j}-f^{k+1}_{j}=\lambda_k f^k_j[\frac{\alpha_j}{p_j}\frac{\partial \Phi(\bff^k)}{\partial f_j}+\cO(\delta_{k})+\cO(\lambda_k)].
		\end{equation} 
		If $ f^{k,i-1}_j\in[U/2,U) $ for all $ i\in\bN_M,$ then
		\begin{equation}\label{eq:decreasing}
			f^{k}_{j}-f^{k+1}_{j}=\lambda_k(U-f^k_{j})[\frac{\alpha_j}{p_j}\frac{\partial \Phi(\bff^k)}{\partial f_j}+\cO(\delta_{k})+\cO(\lambda_k)].
		\end{equation}
	\end{lemma}\par
	We recall that a cluster point of a sequence $ \bff^k $ is defined as the limit of a convergent subsequence of $\bff^k$ and state a lemma whose proof is included in Appendix \ref{appendix:lemconvergence}.
	\begin{lemma}\label{lem:convergence}
		 If conditions (i)-(vi) are satisfied and $\tilde{\bff}^{k,i}\in\mathrm{int\,}\cB$ for all $ k\in\bN_0, i\in\bN_M,$ then (a) $\Phi(\bff^k)$ converges in $\bR$, (b) there exists a cluster point $\bff^*$ of $ \bff^k$ with $\bfS(\bff^*)\nabla\Phi(\bff^*)={\bf0}$, and (c) such a cluster point $\bff^*$ described in (b) is a global minimizer of $\Phi$ over $\cB$.
	\end{lemma}\par

	We are ready to prove convergence of $ \bff^k $ and $\bff^{k,i}$ with the interior assumption. 
	
	\begin{prop}\label{prop:convergence}
		If conditions (i)-(vi) are satisfied and $\tilde{\bff}^{k,i}\in\mathrm{int}\,\cB~\text{for~all~} k\in\bN_0,i\in\bN_M,$ then both
		$\bff^k$ and $\bff^{k,i}$ converge to the global minimizer of $\Phi$ on $\cB$.
	\end{prop}
	\begin{proof} According to Lemma \ref{lem:convergence} (c), $\bff^*$ is a global minimizer of $\Phi$ over $\cB.$ Suppose there exists another cluster point $\bff^{**}\neq\bff^*$. By Lemma \ref{lem:convergence} (a), $\Phi(\bff^k)$ converges in $\bR,$ which implies that $\Phi(\bff^*)=\Phi(\bff^{**}).$ Then $\bff^{**}$ is also a minimizer of $\Phi(\bff),$ which is a contradiction since $\Phi(\bff)$ has a unique minimizer on $\cB.$ Then we obtain $\lim_{k\to\infty}\bff^{k}=\argmin_{\bff\in\cB}\Phi(\bff).$ The convergence of $\bff^{k,i}$ follows from Lemma \ref{lem:difference} and the convergence of $\bff^k$.
	\end{proof}
	
	We have shown that condition $ \tilde{\bff}^{k,i}\in\mathrm{int}\,\cB $ for all $ k\in\bN_0,i\in\bN_M $ is sufficient for the convergence of SDP-BSREM algorithm. Next, we prove the convergence of SDP-BSREM without the interior assumption. The proof of convergence is completed by proving $\tilde{\bff}^{k,i}\in\mathrm{int}\,\cB$ for all $i\in\bN_M$ and $k > K$ for some $K>0.$ We now state a lemma to prove it.

	\begin{lemma}\label{lemma:interior}
		Suppose condition (iii) is satisfied. If $ \bfalpha^{k,i} $ is bounded and $ \lim_{k\to\infty}\lambda_k=0, $ then $\tilde{\bff}^{k,i}\in\mathrm{int}\,\cB $ for all $i\in\bN_M$ and $k>K$ for some $K>0$.
	\end{lemma}
	\begin{proof}
		It suffices to prove that $ \tilde{f}^{k,i}_j \in(0,U) $ for all $ i\in\bN_M, j\in\bN_q, k>K, $ for some $ K>0 $. By condition (iii), $(\partial/\partial f_j)\Phi_i(\bff)$ is bounded over $ \cB $ for all $i\in\bN_M, j\in\bN_q.$  Combining this with the boundedness of $\bfalpha^{k,i}$,  there exists $c_1>0$ such that $|\alpha^{k,i}_j/{p_j}(\partial/\partial f_j)\Phi_i(\bff)|\leqslant c_1$, for all $i\in\bN_M,j\in\bN_q,k\in\bN_0$, and all $\bff\in\cB.$ Because $ \lim_{k\to\infty}\lambda_k=0$, there exists $K>0$ such that $|\lambda_k|<1/c_1 $ for all $k>K$, so that $|\lambda_k\alpha^{k,i}_j/p_j(\partial/\partial f_j)\Phi_i(\bff^{k,i-1})|<1$. Hence, for $k>K$, $i\in\bN_M$, if $f^{k,i-1}_j\in(0,U/2),$ the preconditioner  $S^{k,i}(\bff^{k,i-1})_{jj}=\alpha^{k,i}_jf^{k,i-1}_j/p_j$ gives rise to 
		$
		\tilde{f}^{k,i}_j=f^{k,i-1}_j[1-\lambda_k\alpha^{k,i}_j/p_j(\partial/\partial f_j)\Phi_i(\bff^{k,i-1})]$, from which we can show that $\tilde{f}^{k,i}_j\in(0,U)$. 
		Likewise, if $ f^{k,i-1}_j\in[U/2,U),$ the preconditioner  $S(\bff^{k,i-1})_{jj}=\alpha^{k,i}_j(U-f^{k,i-1}_j)/p_j$ gives that $ (U-\tilde{f}^{k,i}_j)=(U-f^{k,i-1}_j)[1+\lambda_k\alpha^{k,i}_j/p_j(\partial/\partial f_j)\Phi_i(\bff^{k,i-1})] $, from which we can show that $\tilde{f}^{k,i}_j\in(0,U).$
	\end{proof}
	
	We now arrive at the following theorem for the convergence of SDP-BSREM algorithm without an interior assumption.
	\begin{thm}\label{thm:convergenceSDP}
		If conditions (i)-(vi) are satisfied, then $\bff^{k,i}$ converges to the global minimizer of $ \Phi $ on $ \cB. $
	\end{thm}
	\begin{proof}
		We have that $ \lim_{k\to\infty}\lambda_k=0 $ and $ \bfalpha^{k,i} $ is bounded from conditions (iv) and (v) respectively. Thus, by Lemma \ref{lemma:interior}, there exists $ K>0 $ such that $\tilde{\bff}^{k,i}\in\mathrm{int}\,\cB$ for all $k>K, i\in\bN_M$. Then the proof follows from Proposition \ref{prop:convergence}.
	\end{proof}
	
	We next propose four specific SDPs which satisfy the convergence conditions (iv)-(vi). For $\bfS^{k,i}(\bff)=\diag(\bfalpha^{k,i})\bfS(\bff)$, let $ \bfalpha^{k,i}\coloneqq \alpha_{k,i}\bfnu^{k,i} $, where $\alpha_{k,i}$ is a scalar sequence and $\bfnu^{k,i}$ is a vector sequence to be determined. Other, potentially better, choices of $ \bfalpha^{k,i} $ are left as future work. 
	
	In this case, inspired by momentum techniques \cite{Nesterov1983Method,lin2019krasnoselskii}, we consider the following two choices of $ \alpha_{k,i} $.
	The first one is derived from Nesterov momentum \cite{Nesterov1983Method}:
	\begin{equation}\label{preconditioner:Nesterov}
	\alpha_{k,i}\!\coloneqq 1+(t_{k,i}-1)/t_{k,i+1},
	\end{equation} 
	where $ t_{k,i+1}\!\coloneqq (1+\sqrt{1+4t_{k,i}^2})/2, $ $t_{0,1}\coloneqq 1$ and $ t_{k+1,1}\coloneqq t_{k,M+	1}, k\in\bN_0, i\in\bN_M. $ 
	The second one has the following form:
	\begin{equation}\label{preconditioner:KM}
	\alpha_{k,i}\coloneqq   (\varrho(kM+i-1)+\delta_2)/(kM+i-1+\delta_1), 
	\end{equation}
	for $ k\in\bN_0,i\in\bN_M, $ where $ \varrho,\delta_1 $ and $ \delta_2 $ are positive parameters. We notice that this $ \alpha_{k,i} $ is an extension of the momentum proposed in \cite{lin2019krasnoselskii}.\par
	
	The motivation for the design of $\bfnu^{k,i}$  is presented as follows. The use of different step-sizes for different regions in the image can accelerate the convergence of the algorithm.  The diagonal nonnegative definite preconditioner plays an important role in rescaling the step-sizes of the algorithm. Hence, a good preconditioner can significantly accelerate the convergence of the algorithm.  We propose a type of preconditioner that is related to the regularization term that promotes smoothness in the reconstructed image. Our goal is to find the minimizer of the objective function which consists of the fidelity term and the regularization term. The fidelity term estimates the fitting quality of the reconstructed image to the data and the regularization term defined in \eqref{eq:RDP} promotes smoothness in the reconstructed image. Smaller fidelity term makes the reconstructed image more consistent with the data and smaller regularization term leads to a smoother reconstructed image.

	Suppose $ \bff^* $ is a minimizer of the objective function $ \Phi $ and the iteration scheme of the algorithm converges to $ \bff^* $. Let $ I_s $ and $ I_v $ be the smooth and variable areas of $ \bff^* $, respectively. Suppose $ \nabla_{I_s}\Phi $ and $ \nabla_{I_v}\Phi $ are the subsets of $ \nabla\Phi $ defined in the areas $ I_s$ and $ I_v $, respectively. Then for the smooth areas $ I_s $, the descent directions of the fidelity term and the regularization term are consistent whereas for the variable areas $ I_v $, the descent directions of the fidelity term and the regularization term are inconsistent. Thus the direction of $ \nabla_{I_s}\Phi $ more accurately points toward the minimizer than the direction of $ \nabla_{I_v}\Phi. $ Therefore, we conclude that the directions of $ \nabla_{I_s}\Phi(\bff^{k,i}) $ and $ \nabla_{I_s}\Phi(\bff^{k,i+1}) $ are more consistent than the directions of $ \nabla_{I_v}\Phi(\bff^{k,i}) $ and $ \nabla_{I_v}\Phi(\bff^{k,i+1}). $ We illustrate this conjecture (the subset gradient is used in practice) via numerical experiments (see Fig. \ref{fig:Angle}). The preconditioner is designed to achieve larger iteration step-sizes in the smooth areas $ I_s $ and smaller iteration step-sizes in the variable areas $ I_v. $ The numerical gradient (the \emph{gradient} function in Matlab) of $ \bff^* $ is applied to measure the smoothness degree of the image $ \bff^* $. Then larger and smaller step-sizes will be used in the areas having smaller and larger numerical gradients, respectively. 
	
	Suppose $ \bff\in\bR^{q_1q_2\times 1}_+ $ is a 2D image with size $ q_1\times q_2 $. Let $ \mathrm{mat}(\bff)\in\bR^{q_1\times q_2}_+ $ be the matrix form of $ \bff $. Using the \emph{gradient} function in Matlab, we compute the gradients of $ \mathrm{mat}(\bff) $ along the $ x $ and $ y $ directions, namely, $ \mathrm{grad}_x(\mathrm{mat}(\bff)) $ and $ \mathrm{grad}_y(\mathrm{mat}(\bff)) $. Let 
	$
	\mathrm{grad}(\bff)\coloneqq \sqrt{(\mathrm{grad}_x(\mathrm{mat}(\bff)))^2+(\mathrm{grad}_y(\mathrm{mat}(\bff)))^2},
	$
	where the square and square root operations are element-wise. For PET patients data, the minimizer $ \bff^*$ is unknown and we use $ \bff^{k,i} $ to approximate $ \bff^*. $  Based on the consideration that areas with larger numerical gradients should have smaller step-sizes, we first let 
	$
	\bfmu^{k,i}\coloneqq \mathrm{max} (0.01,\mathrm{grad}(\bff^{k,i})/\mathrm{mean}(\bff^{k,i})),
	$
	where $ \mathrm{mean}(\bff)\coloneqq \sum_{j=1}^{q}f_j/q $ is used to normalize the $\bff^{k,i}$.  Instead of directly letting $\bfnu^{k,i}$ be the $1/\bfmu^{k,i}$, we define a projection operator to avoid too large or too small step-sizes. For two positive numbers $ \nu_m<\nu_M, $ and $ \bff\in\bR^q_+, $ a projection operator
	$ P_{\nu_m}^{\nu_M} : \bR^q\to\bR^q $ is defined by $ P_{\nu_m}^{\nu_M}(\bff)_j\coloneqq \min\left\lbrace \nu_M, \max\left\lbrace f_j, \nu_m\right\rbrace\right\rbrace $.

	Let $ J_{k,i}\coloneqq  kM+i, k\in\bN_0, i\in\bN_M $. For $ 0<\nu_1<\nu_2,$ and $ 0\leqslant J_0\leqslant J_1\coloneqq  k_1M+i_1,$ the $\bfnu^{k,i}$ is determined by 
	
	\begin{equation}\label{eq:nuki}
	\bfnu^{k,i}\coloneqq \begin{dcases}
	{\bf1}_q & \text{if~} J_{k,i}\leqslant J_0,\\
	P_{\nu_1}^{\nu_2}(\mathrm{mean}(\bfmu^{k,i})/\bfmu^{k,i}) & \text{if~} J_0<J_{k,i}\leqslant J_1,\\
	\bfnu^{k_1,i_1} &\text{if~} J_{k,i}>J_1.
	\end{dcases}
	\end{equation}
	
	For the first $J_0$ subiterations, the $\bfnu^{k,i}$ is set to the identity vector since the approximation of $\bff^{k,i}$ to $\bff^*$ is poor. The approximation becomes better as the iteration continues, hence different step-sizes for different regions of the image are used for $J_0<J_{k,i}\leqslant J_1$. After $J_1$-th subiteration $\bfnu^{k,i}$ is then fixed due to improved approximation. 
	The preconditioners $\bfS^{k,i}(\bff)=\bfS(\bff)\diag{(\alpha_{k,i}\bfnu^{k,i})}$ are denoted  by P1 and P2 depending on $\alpha_{k,i}$ defined in \eqref{preconditioner:Nesterov} and \eqref{preconditioner:KM} respectively. The momentum-like preconditioners $\bfS^{k,i}(\bff)=\bfS(\bff)\diag{(\alpha_{k,i}{\bf1}_q)}$ are denoted by M1 and M2 depending on $\alpha_{k,i}$ defined in \eqref{preconditioner:Nesterov} and $\eqref{preconditioner:KM}$ respectively. Then we have the following theorem. The proof can be found in the appendix \ref{appendix:lemconvergence}.
	\begin{thm}\label{thm:SpecificSDPs}
		The SDP-BSREM algorithm with $ \Phi $ defined in \eqref{eq:ObjectiveFunction}, and with relaxation $ \lambda_k\coloneqq \lambda_0/(ak+1), \lambda_0, a>0, $ and preconditioner P1, P2, M1, or M2 is convergent.
	\end{thm}
	
	\section{Numerical results}
	\label{sec:NumRes}
	
	In this section, we present results of evaluations of the SDP-BSREM algorithms performance obtained by means of numerical experiments using both simulated and clinical PET data, in comparison with BSREM and with the clinical version of BSREM (Q.Clear, GE).

	\subsection{Simulation Setup}
	
	The algorithms were implemented using a 2D PET simulation model developed in the Matlab environment \cite{schmidtlein2017relaxed,lin2019krasnoselskii}.  The projection matrix, based on a single detector ring of a GE D710 PET/CT, was built using a ray-driven model with 32 parallel rays per detector pair. Cylindrical detector ring, consisting of 576 detectors whose width are 4 mm, was applied. The field of view (FOV) was set to 300 mm and 288 projection angles were used to reconstruct a 256$\times$256 image with pixel size 1.17 mm$\times$1.17mm. The true count projection data were obtained by forward projecting the phantom convolved in image space with an idealized point spread function (PSF). The PSF was a shift-invariant Gaussian function with full width at half maximum (FWHM) equal to 6.59 mm \cite{moses2011fundamental}. Uniform water attenuation, with attenuation coefficient 0.096 $\mathrm{cm}^{-1}$, was simulated using the PET image as support. Scatter was simulated by adding highly smoothed and scaled projection of the phantom to the attenuated image sinograms.  The scaling factor was equal to the estimated scatter fraction $SF\coloneqq  S/(T+S)$, where $T $ and $S$ are true and scatter counts respectively \cite{berthon2015petstep}. Random counts were simulated by adding a uniform distribution to the true and scatter count distributions scaled by a random fraction $RF\coloneqq R/(T+S+R)$, where $R$ is the random count \cite{berthon2015petstep}. The total count was defined as $TC\coloneqq T+S+R.$ In the simulations, it was $ TC=6.8\times 10^6 $ for high count data and $TC=6.8\times 10^5$ for low count data. In both cases $ SF=RF=0.25.$ The individual noise realizations were generated by adding the Poisson noise to the total count distribution. The same system matrix was used to simulate the data and to reconstruct them.
	
	To investigate the convergence acceleration and its impact on reconstructed images fidelity, two figures-of-merit were computed: the objective function ($ \Phi(\bff^k) $) and the normalized root mean square difference (NRMSD).  The region of interest (ROI) based NRMSD is defined by $ \sqrt{\sum_{j\in\Omega}(f_j^k-f_j^{\infty})^2}/\sqrt{\sum_{j\in\Omega}{(f_j^\infty})^2}, $ where $ \bff^{\infty} $ is the converged image at 1000 iterations by BSREM algorithm with 24 subsets in simulations, and $ \Omega $ is the ROI. In the simulations the global NRMSD is obtained by setting the ROI as the whole image.\par
	\begin{figure}[htbp!]
		\centering
		\setcounter {subfigure} {0} a){
			\includegraphics[width=2.5cm]{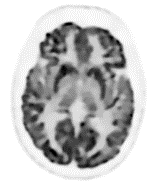}
		}
		\qquad
		\setcounter {subfigure} {0} b){
			\includegraphics[width=3.0cm]{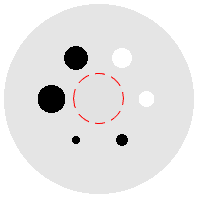}
		}
		\caption{Numerical phantom used in simulations.  a) Brain phantom. b) Uniform phantom: uniform background (1 ROI with radius 25 pixels is shown) with 6 uniform spheres of different radii (2 cold spheres and 4 hot spheres).}
		\label{fig:original image_simulated}
	\end{figure}
    Two $256\times256$ numerical 2D phantoms shown in Fig. \ref{fig:original image_simulated} were used in simulations. The brain phantom \cite{schmidtlein2017relaxed} was obtained from a high quality clinical PET image. The uniform phantom consists of 4 uniform hot spheres and 2 uniform cold spheres with distinct radii: 4, 6, 8 (cold), 10 (cold), 12, 14 pixels. The contrast ratio for the cold and hot spheres are $ 0:1 $ and $ 1:10 $, respectively. All simulations were performed in a 64-bit Windows 10 OS laptop with Intel Core i7-8550U Processor at 1.80 GHz, 16 GB of DDR4 memory and 512 GB SATA SSD.\par
	The parameter $t$ in $P_t$ was set to $10^{-4}$. The constant $\epsilon$ was set to $10^{-12}.$ The regularization parameter $ \beta $ in model \eqref{model-RDP} was set to $ 0.1 $ and $0.8$ for high and low count data respectively. In RDP regularization term, the parameter $\gamma_R$ was set to $2$ and $8$-point neighborhood was considered. The initialization $ \bff^0 $ was set to $ {\bf1}_q $ to examine the setting of $\bfnu^{k,i}$. We used the relaxation sequence defined by $ \lambda_0/(ak+1), a>0 $ . In all simulation experiments, for simplicity, we empirically set $\lambda_0=1, J_0=3, J_2=1000 $ and $ \delta_1=\delta_2. $  Other parameter values, shown in Table \ref{tab:AlgorithmParameter_2D}, were chosen based on the performance of objective function value.\par

	\begin{table}[!htbp]
		\centering
		\caption{Algorithmic parameters for 2D simulation reconstruction}
		\label{tab:AlgorithmParameter_2D}
		\begin{tabular}{|c|c|c|}
			\hline
			Parameters & Brain phantom & Uniform phantom  \\ \hline
			BSREM(12)  & \begin{tabular}[c]{@{}c@{}} high count: $a=1/400$\\ low count: $a=1/18$\end{tabular} & - \\ \hline
			SDP-P1(12) & \begin{tabular}[c]{@{}c@{}} high count: $a=1/13$,\\ $\nu_1=1.6,\nu_2=2.4$\\ low count: $a=0.5$,\\ 	$\nu_1=1.6,\nu_2=2.4$\end{tabular} 
			& -  \\ \hline
			SDP-P2(12) & \begin{tabular}[c]{@{}c@{}} high count: $a=1/5,\varrho=5$,\\ $\delta_1=5,\nu_1=0.8,\nu_2=2.2$\\ low count: 	$a=1.3,\varrho=7.5$,\\ $\delta_1=5,\nu_1=1.3,\nu_2=2.1$\end{tabular} 
			& - \\ \hline
			BSREM(24)  & \begin{tabular}[c]{@{}c@{}} high count: $a=1/35$\\ low count: $a=1/5$\end{tabular} & \begin{tabular}[c]{@{}c@{}} high count:\\ $a=1/35$ \end{tabular}\\ \hline
			SDP-P1(24) & \begin{tabular}[c]{@{}c@{}} high count: $a=0.35$,\\ $\nu_1=1.6,\nu_2=2.4$\\ low count: $a=1.3$,\\ 	$\nu_1=1.4,\nu_2=2.5$\end{tabular} 
			& \begin{tabular}[c]{@{}c@{}} high count: $a=0.5$,\\ $\nu_1=1.8,\nu_2=2.5$\end{tabular}  \\ \hline
			SDP-P2(24) & \begin{tabular}[c]{@{}c@{}} high count: $a=0.45,\varrho=4$,\\ $\delta_1=3,\nu_1=0.8,\nu_2=1.8$\\ low count: 	$a=1.4,\varrho=2.2$,\\ $\delta_1=1,\nu_1=1.3,\nu_2=2.4$\end{tabular} 
			& \begin{tabular}[c]{@{}c@{}} high count: $a=0.7$, \\ $\varrho=3,\delta_1=7$,\\ $\nu_1=1.4,\nu_2=2.3$\end{tabular} \\ \hline
		\end{tabular}
	\end{table}

	\subsection{Simulation Results}
	
	\subsubsection{Comparison of Gradient Consistency}
	To measure the directional consistency of two vectors, we computed the angle between them. The angle between vectors $ \bfv_1 $ and $ \bfv_2 $  is defined as $\theta(\bfv_1,\bfv_2)\coloneqq \arccos (\left\langle\bfv_1,\bfv_2\right\rangle / (\|\bfv_1\|_2\|\bfv_2\|_2)).
	$
	We define the smooth areas sequence by 
	$$
	I_s^{k,i}\coloneqq \left\lbrace j\in\bN_q: \mathrm{grad}(\bff^{k,i})_j<0.01\cdot\mathrm{mean}(\bff^{k,i}) \right\rbrace,
	$$
	and the variable areas sequence by
	$$
	I_v^{k,i}\coloneqq \left\lbrace j\in\bN_q: \mathrm{grad}(\bff^{k,i})_j>0.2\cdot\mathrm{mean}(\bff^{k,i}) \right\rbrace.
	$$
	In order to estimate the consistency of $ \nabla_{I_s^{k,i}}\Phi_{i-1}(\bff^{k,i-1}) $ and $ \nabla_{I_s^{k,i}}\Phi_i(\bff^{k,i}) $, we computed the angle  between them. For smooth areas, we define the angle sequence by 
	$$\theta_{k,i}\coloneqq \theta(\nabla_{I_s^{k,i}}\Phi_{i-1}(\bff^{k,i-1}),\nabla_{I_s^{k,i}}\Phi_i(\bff^{k,i})),$$
	and the average angle in each iteration by
		$
		\theta_k\coloneqq \sum_{i=1}^{M}\theta_{k,i}/M,
		$
		where $ \Phi_0\coloneqq \Phi_M. $ Similarly, for variable areas, we computed 
		$$ \tilde{\theta}_{k,i}\coloneqq \theta(\nabla_{I_v^{k,i}}\Phi_{i-1}(\bff^{k,i-1}),\nabla_{I_v^{k,i}}\Phi_i(\bff^{k,i})) $$ and $ \tilde{\theta}_k\coloneqq \sum_{i=1}^{M}\tilde{\theta}_{k,i}/M. $
	\begin{figure}[htbp!]
		\centering
		\subfigure{
			\includegraphics[width=0.55\textwidth]{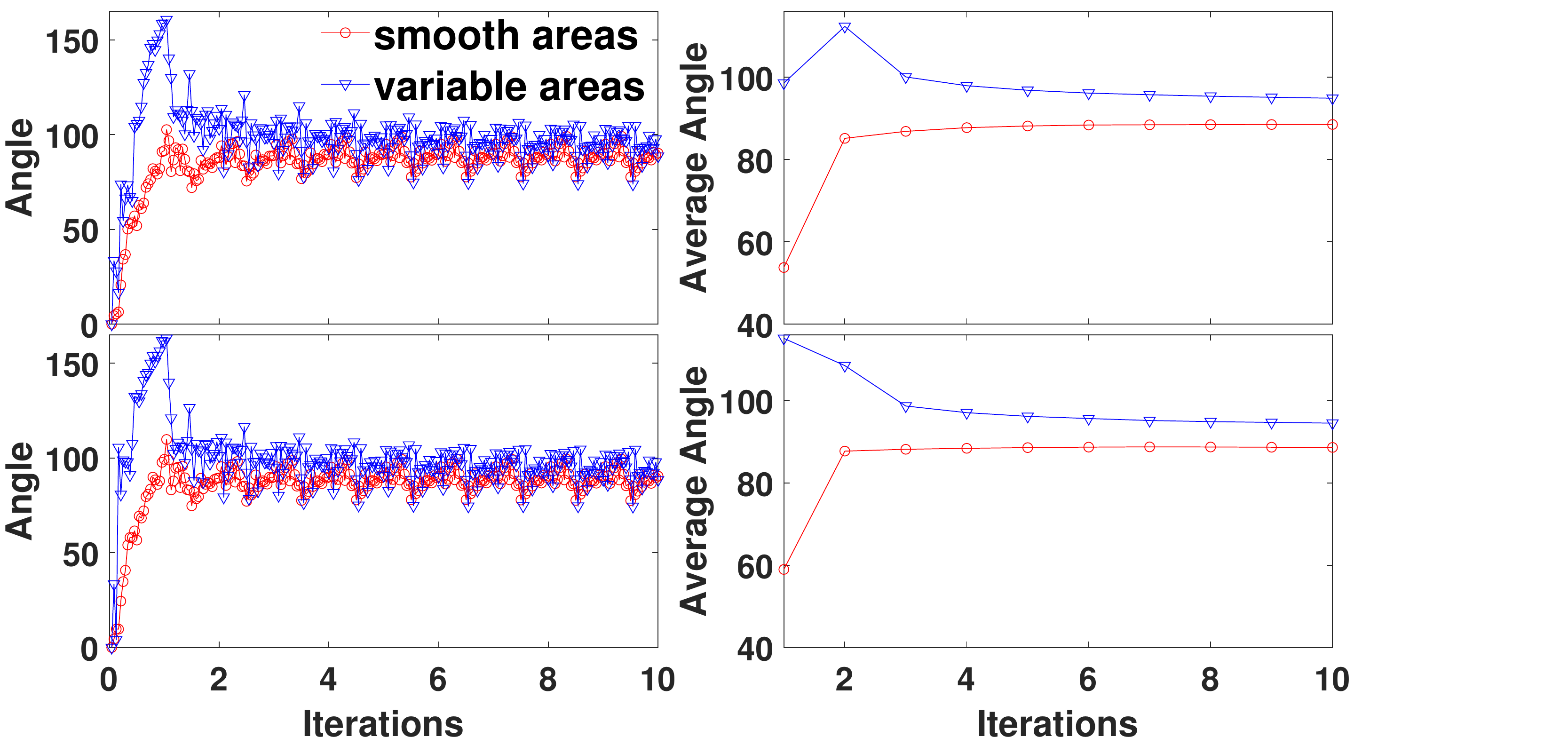}
		}
		\caption{Angle (left) and average angle (right) between the gradients of the successive subiterations vs. iterations projected in the smooth areas and variable areas in the reconstructed images, respectively, for the brain phantom with high count data. Top row:  SDP-P1 (24). Bottom row: SDP-P2(24).}
		\label{fig:Angle}
	\end{figure}
	In Fig. \ref{fig:Angle}, we observed that for SDP-P1 and SDP-P2, the angle and average angle in smooth areas were smaller than those in the areas with more variability.  This is consistent with our conjecture that the directions of gradients in the smooth areas more accurately point toward the minimizer than those in the variable areas. Hence, larger step-sizes in the smooth areas and smaller step-sizes in the variable areas are reasonable.
	
	\subsubsection{Comparison of Preconditioners}
	In order to reveal the improvement due to the application of a preconditioner, as compared to the use of a momentum, we compared SDP-BSREM algorithm with four different preconditioners: P1, P2, M1, and M2, where M1 and M2 are surrogates of momenta. The parameter $a$ in SDP-M1(12) and SDP-M1(24) was set to $1/50$ and $1/6$, respectively. And the parameters $a, \varrho, \delta_1$ in SDP-M2(12) and SDP-M2(24) were set to $1/15,3,1$ and $1/5,2.6,0.5$, respectively. In Fig. \ref{fig:ComparePreconditioner}, one can observe that SDP-P1 and SDP-P2 outperform SDP-M1 and SDP-M2, in reaching the same objective function value, by 25-30\% and 25\%, respectively.\par
	
	\begin{figure}[htbp!]
		\centering
		\subfigure{
			\includegraphics[width=0.48\textwidth]{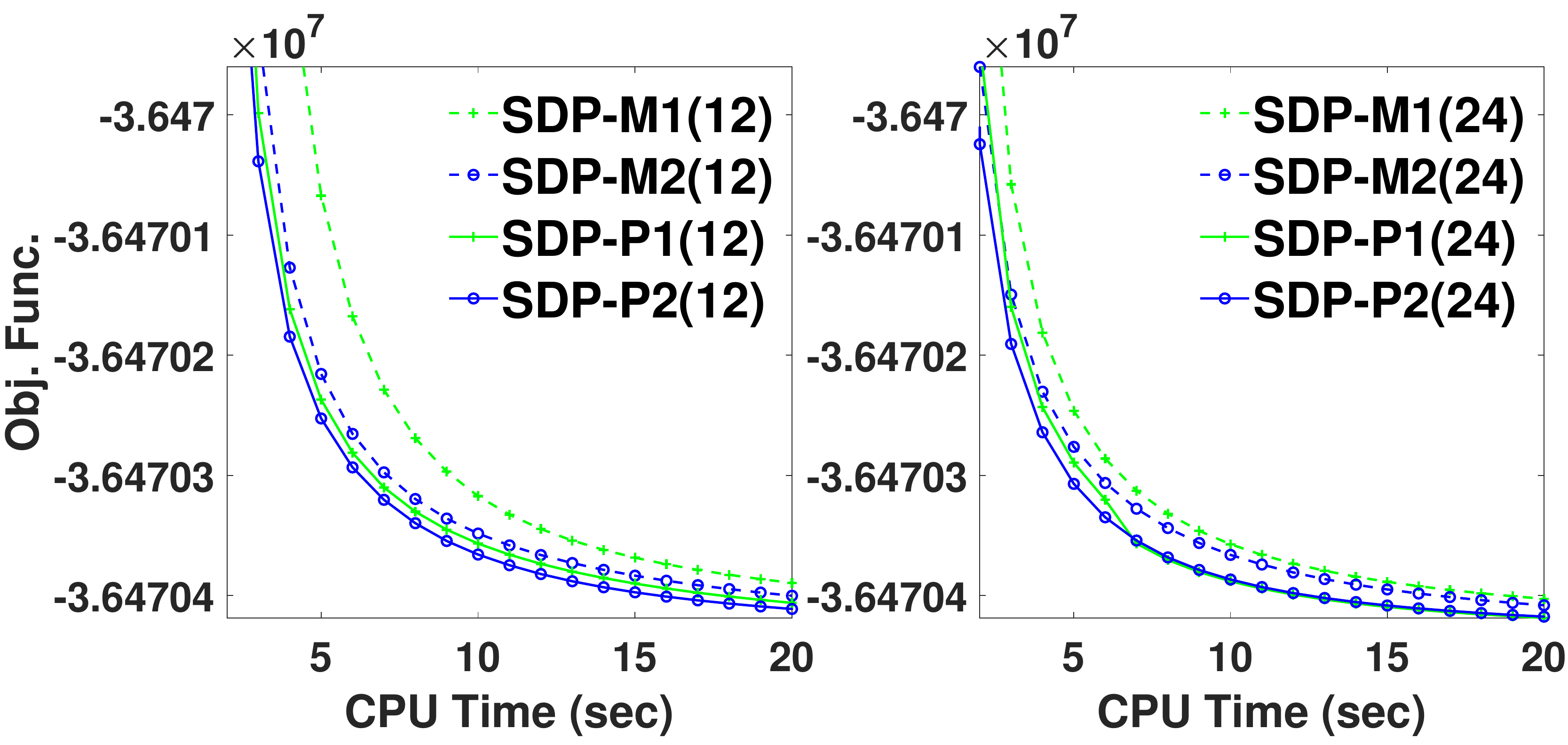}
		}
		\caption{Comparison of performance of preconditioners investigated in this study. Objective function vs. elapsed CPU time in reconstructions performed with SDP-BSREM algorithm with four preconditioners: M1, M2, P1, and P2, and with  12 subsets (left) and 24 subsets (right), respectively, for the brain phantom with high count data. Preconditioners P1 and P2 were generalized from M1 and M2, respectively.}
		\label{fig:ComparePreconditioner}
	\end{figure}

	\subsubsection{Comparison of SDP-BSREM with BSREM}
	In this subsection, we analyzed the performance of SDP-BSREM algorithms compared to the BSREM algorithm. First, we showed the global NRMSD for all the algorithms, with 12 and 24 subsets, in Fig. \ref{fig:GlobalNRMSD}, using the brain phantom. It showed that all algorithms converged to the same solution for both low and high count data.  Further, this figure showed that SDP-P1 and SDP-P2 outperformed BSREM with respect to global NRMSD. To analyze convergence acceleration, we showed the objective function values of each algorithm in Fig. \ref{fig:OFV_SDP}. In this figure, one can observe that both proposed algorithms, SDP-P1 and SDP-P2, outperform the BSREM algorithm, in reaching the same objective function value, by roughly a factor of two for all cases: 12 and 24 subsets for both low and high count data using the brain phantom.\par
	\begin{figure}[htbp!]
		\centering
		\subfigure{
			\includegraphics[width=0.55\textwidth]{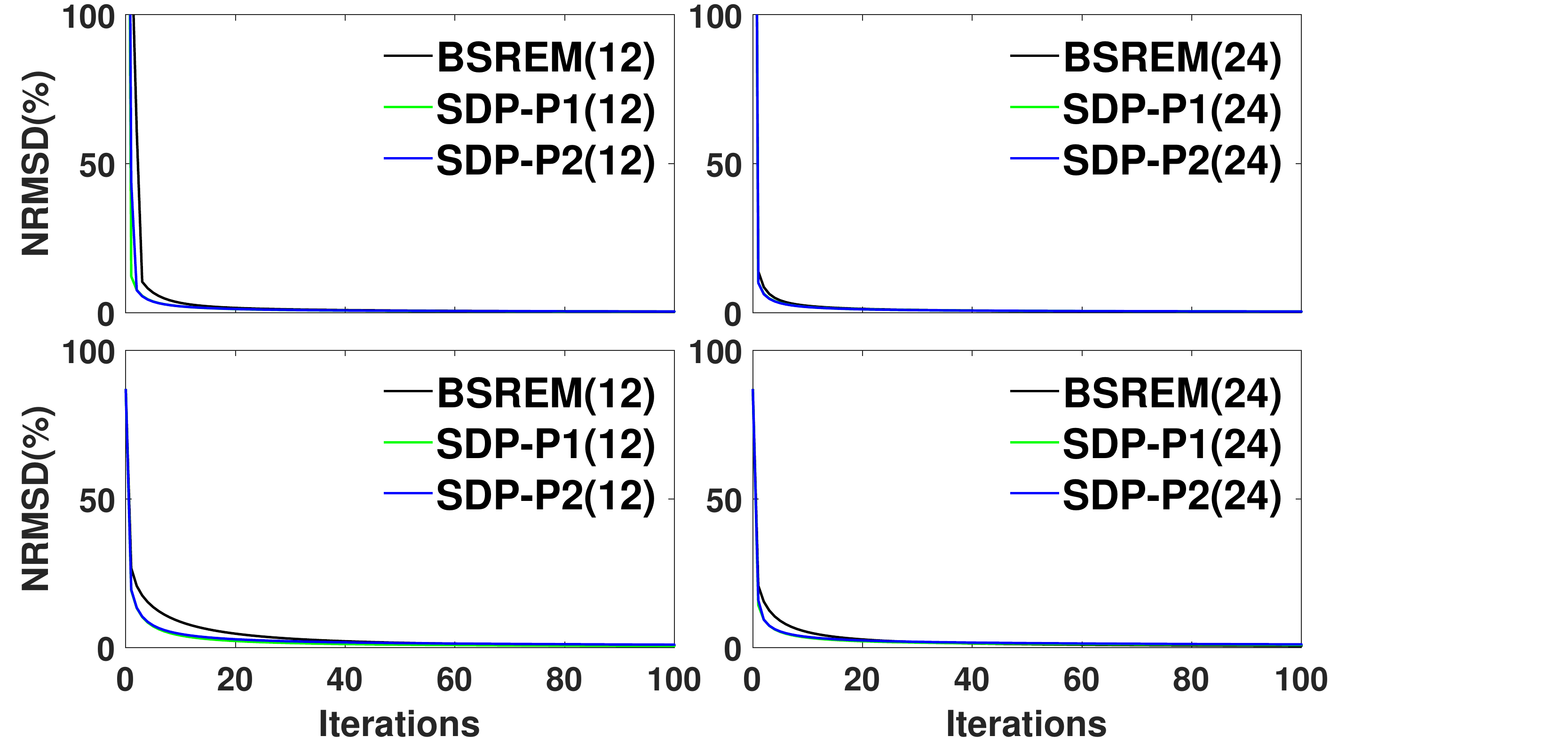}
		}
		\caption{Global NRMSD vs. Iterations in reconstructions performed with different algorithms: BSREM(12), SDP-P1(12), SDP-P2(12), BSREM(24), SDP-P1(24), SDP-P2(24) for the brain phantom with low (top row) and high (bottom row) count data, respectively.}
		\label{fig:GlobalNRMSD}
	\end{figure}
	\begin{figure}[htbp!]
		\centering
		\subfigure{
			\includegraphics[width=0.65\textwidth]{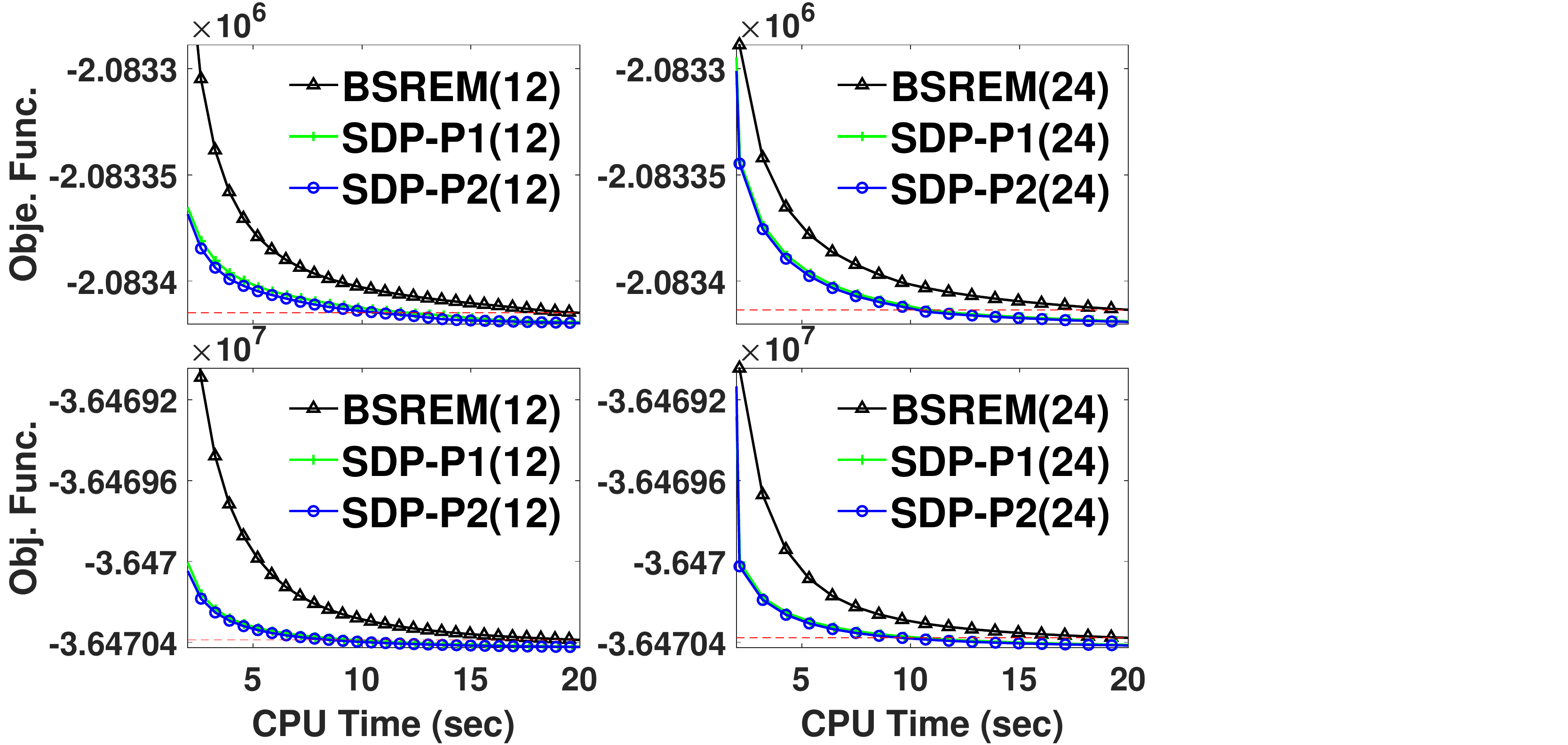}
		}
		\caption{Comparison of performance of SDP-BSREM vs. BSREM algorithm. Objective function vs. elapsed CPU time in reconstructions performed with BSREM, SDP-P1, and SDP-P2, with 12 (left) and 24 (right) subsets for the brain phantom with low (top row) and high (bottom low) count data , respectively. The dash lines represent the objective function values of BSREM at 20 seconds CPU time.}
		\label{fig:OFV_SDP}
	\end{figure}
	\begin{figure}[htbp!]
		\centering
		\subfigure{
			\includegraphics[width=.47\textwidth,left]{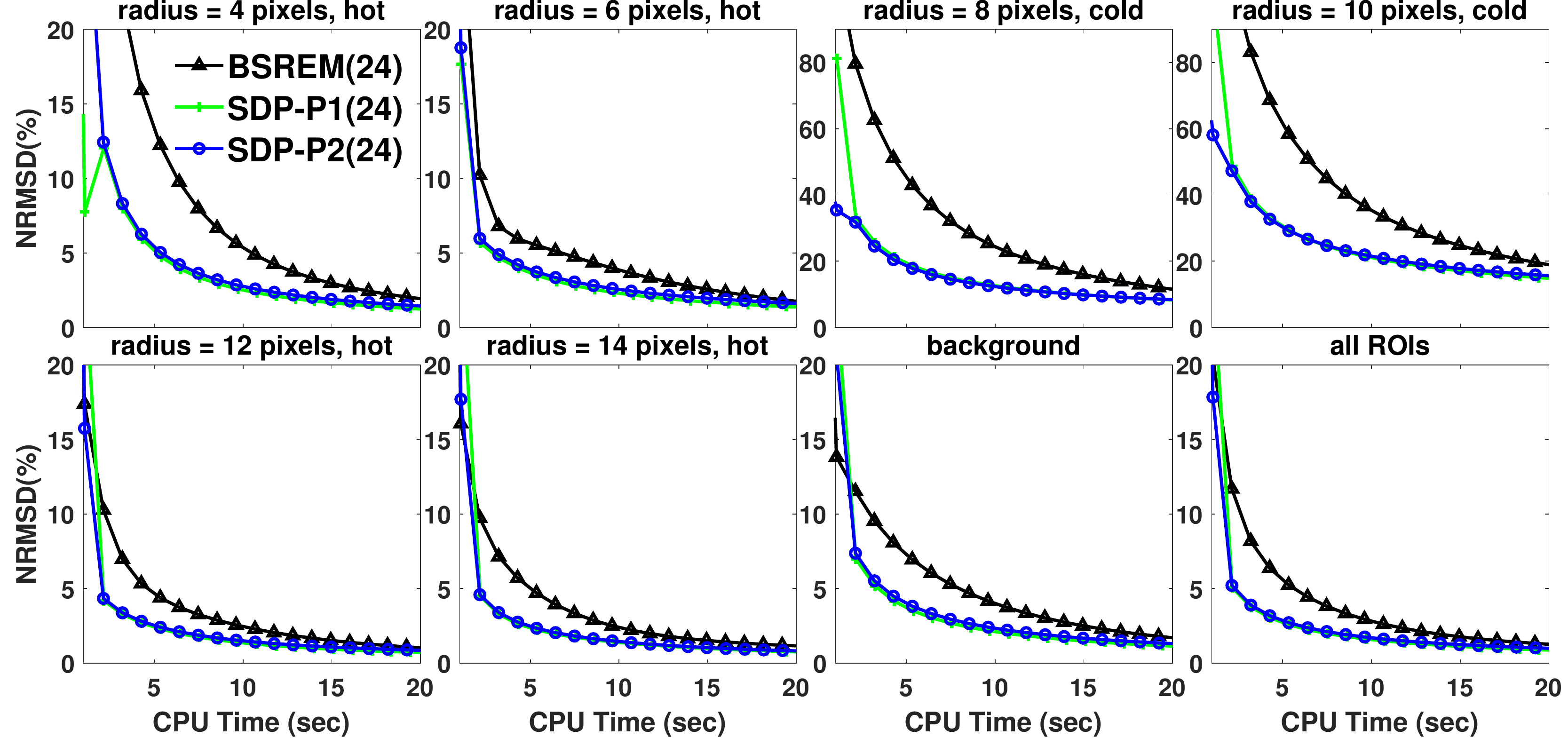}
		}
		\caption{Comparison of local convergence performance of SDP-BSREM vs. BSREM algorithms for the uniform phantom with high count data. ROI based normalized root mean square difference (NRMSD) vs. iterations is shown. The eight ROIs are the 4 hot spheres, 2 cold spheres, 1 background spheres, and the region consisting of all the former 7 ROIs, named ``all ROIs".}
		\label{fig:NRMSDUniformPhantom}
	\end{figure}
	Next, we examined the local convergence performance of SDP-BSREM algorithms by ROI based NRMSD in 8 different ROIs with different contrast ratios in the reconstructions of the uniform phantom. High count data and 24 subset were used in this experiments. In Fig. \ref{fig:NRMSDUniformPhantom}, we observed that the proposed SDP-P1 and SDP-P2 algorithms converged fast than BSREM algorithm in all 8 ROIs.\par
	
	\subsection{Clinical Results}
	{\color{black}The reconstructions were performed with our SDP-BSREM algorithm with P2 preconditioner and with commercial Q.Clear by means of the GE toolbox \cite{GEPETToolbox5.0} with the penalty weight ($\beta$) set to the default value of 350. Because the 2D-projectors used in the simulations and 3D-projectors used in clinical data reconstructions were scaled differently, the penalty values used in respective reconstructions differed substantially.}\par
	
	\begin{figure}[htbp!]
		\centering
		\setcounter {subfigure} {0} a){
			\includegraphics[width=1.7cm]{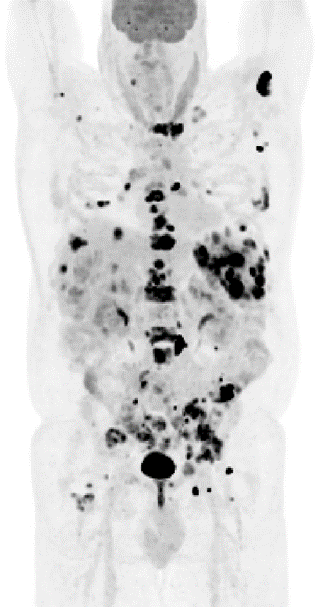}
		}	
	\qquad
		\setcounter {subfigure} {0} b){
			\includegraphics[width=3.3cm]{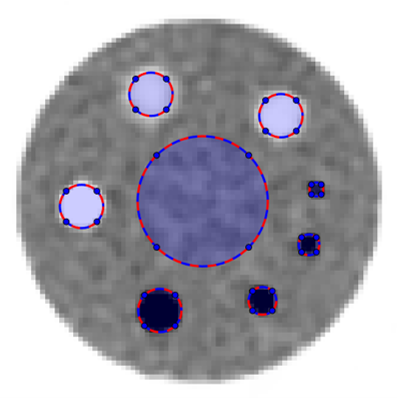}
		}
		\caption{Clinical PET patient and ACR phantom.  a) Clinical PET patient: Coronal maximum intensity projection (MIP) of a clinical whole-body 18F-FDG PET patient image acquired on a GE D710 PET/CT and reconstructed using the Q.Clear clinical method \cite{ahn2015quantitative}. b) Clinical ACR quality assurance phantom showing the regions of interest for cold/hot cylinders, 0:1 and 2.5:1 activity concentration ratios, respectively, and  background \cite{Acr2012pet}.}
		\label{fig:original image_clinical}
	\end{figure}
	To mimic the GE's clinical implementation of Q.Clear, 25 and 8 iterations were used for non-TOF and TOF data, respectively, with 24 subsets in the experiments.  For the same reason we initialized both the non-TOF and TOF reconstructions using OSEM with 2 iterations and 24 subsets.  The reconstructions with TOF data were further initialized using 3 iterations with 24 subsets of non-TOF algorithm.  This gives a more clinically realistic view of the performance, but at the cost of being able to isolate TOF performance. 
	
	The parameter values are shown in Table \ref{tab:AlgorithmParameter_3D}. For simplicity, since good initializations were used, we set $J_0 = 0$ and $ J_1 = 1000$. The other algorithmic parameters were found via an iterative golden search procedure using a single bed position (centered on the Derenzo region) from an ACR PET phantom \cite{macfarlane2006acr} with similar count characteristics as the patient's data.  Using this phantom each parameter was sequentially optimized with 5\% tolerance and then used in search for the next parameter until parameter values ceased to change ($\sim$3 iterations).  
	\begin{table}[!htbp]
		\centering
		\caption{Algorithmic parameters for 3D patient reconstruction}
		\label{tab:AlgorithmParameter_3D}
		\begin{tabular}{|c|c|}
			\hline
			Algorithm  & Parameters    \\ \hline
			Q.Clear(nonTOF) & $\lambda_0=2, a=1/5$  \\ \hline
			Q.Clear(TOF)    & $\lambda_0=1.2, a=1/5$ \\ \hline
			SDP-P2(nonTOF)  & \begin{tabular}[c]{@{}c@{}}$\lambda_0=1.6, a=1/4, \varrho=2.2$, \\ 
			$\delta_1=0.1, \delta_2= 1.6, \nu_1=0.6, \nu_2=1.25$\end{tabular} \\ \hline
			SDP-P2(TOF)     & \begin{tabular}[c]{@{}c@{}}$\lambda_0=1.1, a=1/3, \varrho=2.4$, \\ 
			$\delta_1=0.6, \delta_2= 1.6, \nu_1=0.6, \nu_2=1.25$\end{tabular} \\ \hline
		\end{tabular}
	\end{table}
	
	\begin{figure}[htbp!]
		\centering
		\subfigure{
			\includegraphics[width=.475\textwidth,left]{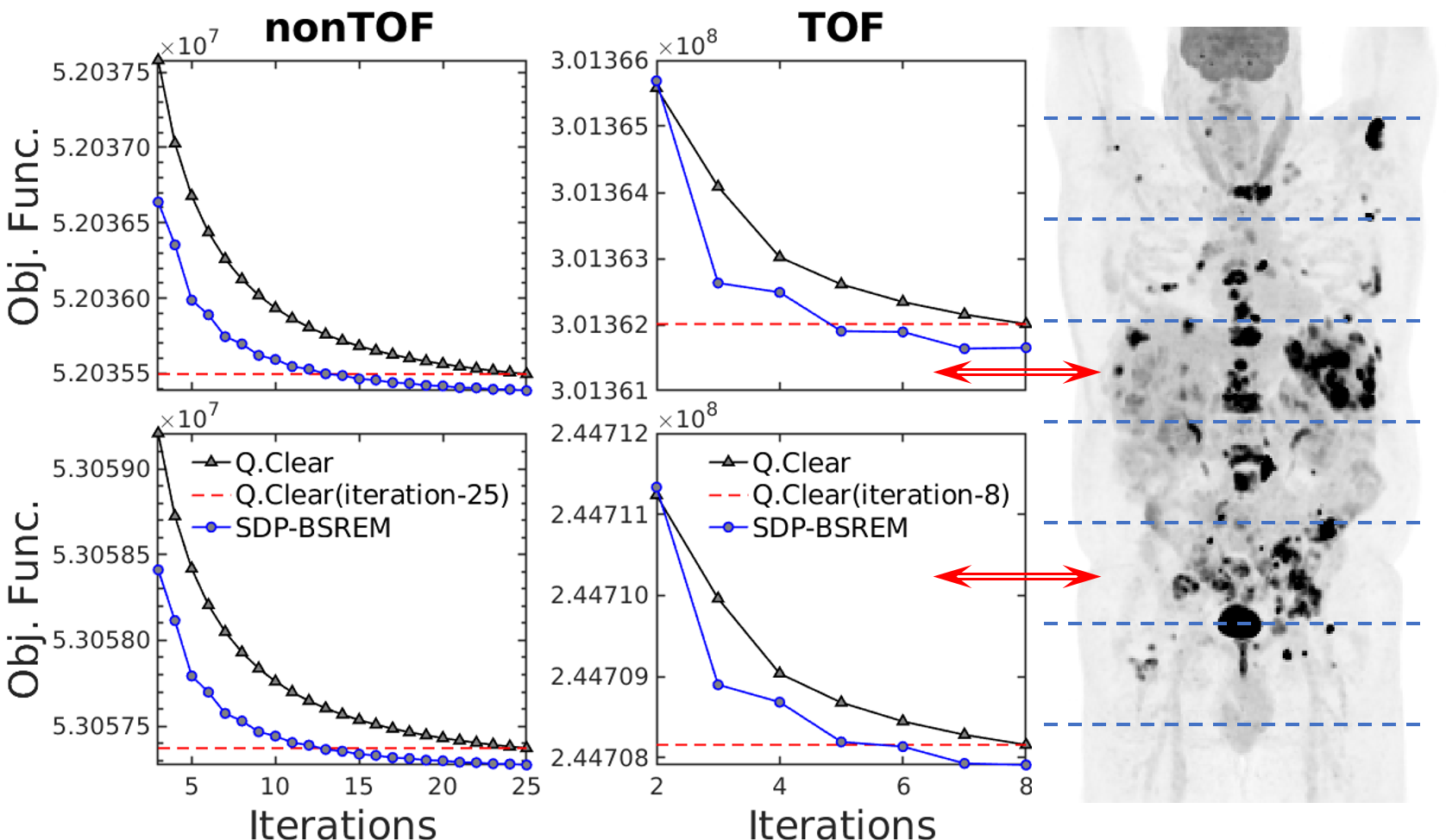}
		}
		\caption{Comparison of performance of SDP-BSREM vs. Q.Clear $(\beta = 350)$ algorithms. A whole-body 18F-FDG clinical PET patient scan was used.  Eight patient bed positions separated by dashed lines are shown in the coronal maximum intensity projection (MIP) image. The objective function vs. iterations is shown for PET scanner patient bed positions 4 and 6 ({\color{red}red arrows}) for nonTOF (left) and TOF data (right). The dashed lines represent the objective function values at the final Q.Clear iterations.}
		\label{fig:clinical}
	\end{figure}
	
	In Fig. \ref{fig:clinical} we show convergence, via the objective function value as a function of iteration for non-TOF/TOF data, for an $^{18}$F-FDG whole-body PET clinical patient (shown in Fig. \ref{fig:original image_clinical}a). This data was obtained from 8 bed positions acquired on a GE D710 PET/CT. The nominal administered activity and post-administration acquisition were 444 MBq and 1-hour,
	respectively, with 3-minute dwell times and 25\% overlap, resulting in $[4.1 / 3.4 / 4.2 / \textbf{4.5} / 4.6 / \textbf{3.8} / 3.1 / 2.6]\times 10^7$ total counts, where the bolded numbers are from the bed positions 4 and 6 shown in Fig. \ref{fig:clinical}.
	
	{\color{black} We observed that our SDP-BSREM method outperformed the Q.Clear algorithm, in reaching the same objective function value, by 40-50\% and 35-50\% for non-TOF and TOF data, respectively. We note that both the clinical 3D projection and penalty operator have much greater computational complexity than the convergence acceleration scheme described in section \ref{section_Convergence_of_SDP-BSREM_algorithm}. Hence, the increased computational cost required for the use of the SDP is negligible.}\par
	
	\begin{figure}[htbp!]
		\centering
		\subfigure{
			\includegraphics[width=.475\textwidth,left]{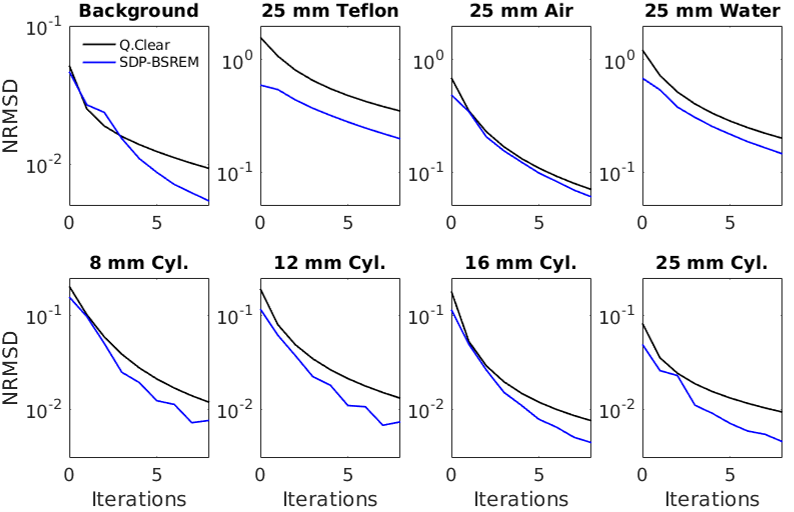}
		}
		\caption{Local convergence is assessed using the 8 regions of interest from an ACR quality assurance test with TOF data \cite{Acr2012pet}. These regions of interest can be seen in Fig. \ref{fig:original image_clinical}b. Each subplot represents one of the eight regions, which are from left to right and top to bottom: background, cold Teflon/air/water, and hot 8/12/16/25 mm cylinders, respectively.}
		\label{fig:NOFV_vs_NRMSD}
	\end{figure}
	To evaluate the local convergence, TOF data from a quarterly ACR quality assurance test was used.  Following ACR guidelines \cite{Acr2012pet} the activity corresponded to a nominal 444 MBq (12 mCi) of $^{18}$F-FDG administration used at MSKCC. The upper proton of this phantom contains 8 regions (3 cold, 4 hot, and background) with nominal contrast ratios of 0:1 and 2.5:1 for the cold and hot cylinders, respectively.  ROI were defined using the cylinder boundaries from registered CT images. Using the methodology described by Kim \etal \cite{kim2019time}, for each ROI we measured the NRMSD,	where the $f^\infty$ in NRMSD is the converged image at 300 iterations by Q.Clear without subsets.  These reconstructions used the same parameters as those used in the whole body patient reconstructions (\ie $J_0 = 0, J_1 = 1000$ and Table \ref{tab:AlgorithmParameter_3D}). The results are shown in Fig. \ref{fig:NOFV_vs_NRMSD}.  For each ROI, the SDP-BSREM method converged to $f^\infty$ faster than Q.Clear.
	
	\section{Conclusion}
	{\color{black}
	In this paper, we have presented the SDP-BSREM algorithms with two SDPs and their global convergence theorems. The two SDPs were designed based on the smoothness-promoting property in the reconstructed images of the regularization term. We tested these algorithms using both simulated and clinical PET data. Using two simulated phantoms, our numerical studies showed that, for solving the RDP regularized PET image reconstruction model, our proposed algorithms converged more quickly than BSREM.  Similarly, when using clinical patient and phantom PET data, our proposed algorithm SDP-P2 outperformed Q.Clear. We plan to test the SDP-BSREM algorithm on more varied data sets acquired under a wide range of conditions seen in the clinic.
	
	\appendices
	\section{}
	\label{appendix:ConvexAndGradient}
	This appendix includes the proof of strict convexity and Lipschitz continuous gradient of the objective function $\Phi$.
	\begin{prop}\label{prop:StrictlyConvex}
		If $\bfA^\top\bfg\neq{\bf0}$, then the objective function $\Phi(\bff)$ in \eqref{eq:ObjectiveFunction} is strictly convex on $\bR^q_+$. 
	\end{prop}
	\begin{proof}
		From \eqref{eq:FidelityTerm}, the gradient of fidelity term $F$ is given by
		$
		\nabla F = \bfA^{\top}({\bf1}_p-\bfg/(\bfA\bff+\bfgamma)).
		$
		Then the Hessian matrix of $F$ can be computed as follows:
		\begin{equation}\label{Hessian_F}
			\nabla^2F(\bff)=\bfA^\top\bfG\bfA
		\end{equation}
		with $\bfG$ a diagonal matrix and
		$\bfG=\diag({\bfg}/{(\bfA\bff+\bfgamma)^2})$.
		The first-order partial derivative of the RDP term $R(\bff)$
		is given by
		\begin{equation}
			\frac{\partial R}{\partial f_j}\! = \!2\!\!\sum_{k \in N_j}\! \frac{(f_j-f_k)(\gamma_{R}|f_j-f_k|+f_j+3f_k+2\epsilon)}{(f_j+f_k+\gamma_{R}|f_j-f_k|+\epsilon)^2}.
			\label{eqn:gradRD}
		\end{equation}
		Then we have the second-order partial derivative of $R$ as follows:
		\begin{equation}\label{eq:SecondDerivative}
			\frac{\partial^2R}{\partial f_j\partial f_k}=\begin{dcases}
				\sum_{l\in N_j}\frac{4  (2f_l+\epsilon)^2}{(f_j+f_l+\gamma_{R}\left|f_j-f_l\right|+\epsilon)^3} &\text{if } k=j,\\
				-\frac{4 \left( 4f_j f_k+2\epsilon(f_j+f_k)+\epsilon^2\right)}{(f_j+f_k+\gamma_{R}\left|f_j-f_k\right|+\epsilon)^3} &\text{if } k\in N_j,\\
				0 &\text{otherwise}.
			\end{dcases}
		\end{equation}
		For any ${\bf0}\neq\bfx\in\bR^q,$
		ignoring the zero entries of $ \nabla^2R, $ we obtain
		\begin{equation}\label{eq:Quadratic}
			\bfx ^\top\nabla^2 R\bfx=\sum_{j=1}^q\frac{\partial^2R}{\partial f_j^2}x_j^2+\sum_{j=1}^q\sum_{k\in N_j}\frac{\partial^2R}{\partial f_j\partial f_k}x_j x_k.
		\end{equation}
		By \eqref{eq:SecondDerivative} and \eqref{eq:Quadratic}, we have
		\begin{equation}\label{Hessian_R}
			\bfx ^\top\nabla^2 R\bfx=\sum_{j=1}^q\!\sum_{k\in N_j}\!\frac{2  \left((2f_k+\epsilon)x_j-(2f_j+\epsilon)x_k\right)^2}{(f_j+f_k+\gamma_{R}\left|f_j-f_k\right|+\epsilon)^3}.
		\end{equation}
		We can see that $\bfx^\top\nabla^2R\bfx\geqslant0$ and for any $ \bfx\neq {\bf0}, \bfx^\top\nabla^2R\bfx=0$ if and only if  there exists a nonzero constant $c,$ such that $\bfx=c(2\bff+\epsilon)$. For any $\bfx=c(2\bff+\epsilon)\neq{\bf0},$ we have
		$\bfx^\top\nabla^2F\bfx=c^2\|\bfG^{1/2}\bfA(2\bff+\epsilon)\|^2,$
		thus $ \bfx^\top\nabla^2F\bfx>0 $ by using $\bfA^\top\bfg\neq{\bf0}.$ Since
		$\nabla^2\Phi(\bff)=\nabla^2F(\bff)+\beta\nabla^2 R(\bff),$	one can obtain that
		$\bfx^\top\nabla^2\Phi\bfx>0$ for all $\bfx\neq{\bf0}.$ Then $\Phi$ is strictly convex on $\bR^q_+.$
	\end{proof}

	\begin{prop}\label{prop:LipGrad}
		The objective function $\Phi(\bff)$ in \eqref{eq:ObjectiveFunction} has a Lipschitz continuous gradient on $ \bR^q_+ $.
	\end{prop}
	
	\begin{proof}
		From \eqref{Hessian_F}, one can obtain that 
		$$
		\|\nabla^2 F\|_2\leqslant{\|A\|_2^2\|g\|_{\infty}}/{\gamma_{\text{min}}},
		$$
		where $ \gamma_{\text{min}}>0 $ is the minimum entry of $ \bfgamma. $ 
		For any $ \bfx\in\bR^q$ with $\|\bfx\|_2=1 $ and $ \bff\in\bR^q_+, $ let $ h(\bff)\coloneqq \bfx ^\top\nabla^2 R(\bff)\bfx $. From \eqref{Hessian_R}, we know that $ h(\bff) $ is continuous on $ \bR^q_+ $ and $ \lim_{\|\bff\|_2\to\infty}h(\bff)=0 $.
		Thus there exists $ C_1>0 $ such that $ \left|h(\bff)\right|\leqslant C_1 $ for any $ \bfx $ with $ \|\bfx\|_2=1 $. Then we have $ \|\nabla^2 R\|_2\leqslant C_1 $ which, when combined with the boundedness of  $\|\nabla^2 F\|_2$, implies that 
		$$ \|\nabla^2\Phi\|_2\leqslant {\|A\|_2^2\|g\|_{\infty}}/{\gamma_{\text{min}}}+C_1. $$ Using Lemma 1.2.2 in \cite{Nesterov2004Introductory}, one can obtain that $ \nabla\Phi(\bff) $ is Lipschitz continuous with constant $ C\coloneqq  (\|A\|_2^2\|g\|_{\infty})/(\gamma_{\text{min}})+C_1 $ on $ \bR^q_+. $ 
	\end{proof}
	
	\section{}
	\label{appendix:lemconvergence}
	This appendix includes proofs of Lemma \ref{lem:LipConPrecond}, Lemma \ref{lem:monotone}, Lemma \ref{lem:convergence}, and Theorem \ref{thm:convergenceSDP}. Here is the proof of Lemma \ref{lem:LipConPrecond}.
	
	\begin{proof}
	We present only a proof of the  Lipschitz continuity and the uniform boundedness can be shown in a similar manner.
		For any $ \bff,\tilde{\bff}\in\cB,$ we consider the quantities
		$
		\Delta:=\bfS^{k,i}(\bff)\nabla\Phi_i(\bff)-\bfS^{k,i}(\tilde{\bff})\nabla\Phi_i(\tilde{\bff}),
		$
		$
		\Delta_1:=\nabla\Phi_i(\bff)-\nabla\Phi_i(\tilde{\bff}), \ \mbox{and} \  \Delta_2:=	\bfS^{k,i}(\bff)-\bfS^{k,i}(\tilde{\bff}).
		$
		
		By the triangle inequality and  the Cauchy–Schwarz inequality, we have that
		\begin{equation}\label{DELTA}
		\|\Delta\|_2
		\leqslant  \|\bfS^{k,i}(\bff)\|_2\|\Delta_1\|_2
		 + \|\nabla\Phi_i(\tilde{\bff})\|_2\|\Delta_2\|_2.
		\end{equation}
		According to condition (v), $\bfalpha^{k,i} $ is bounded. This together with the definition of $ \bfS(\bff)$ implies that $\bfS^{k,i}(\bff)\coloneqq\diag(\bfalpha^{k,i})\bfS(\bff)$ is bounded for all $k\in\bN_0, i\in\bN_M$ and $\bff\in\cB$. By condition (iii), there exists  $ c_2>0$ such that for all $k\in\bN_0,i\in\bN_M$ 
		\begin{equation}\label{c_2}
		\|\bfS^{k,i}(\bff)\|_2\|\Delta_1\|_2\leqslant c_2 \|\bff-\tilde{\bff}\|_2.
		\end{equation}
		We next prove that $ \bfS^{k,i}(\bff) $ are Lipschitz continuous on $ \cB $ with Lipschitz constants bounded above by a constant. For any $ \bff,\tilde{\bff}\in\cB, $ if both $f_j, \tilde{f}_j $ are in $ [0,U/2) $ or in $ [U/2,U],$ then
		$|S^{k,i}(\bff)_{jj}-S^{k,i}(\tilde{\bff})_{jj}| = |\alpha^{k,i}_j/p_j|\cdot |f_j-\tilde{f}_j|$.
		If $ f_j $ and $ \tilde{f}_j $ are not in the same interval, without lose of generality, assuming $ f_j\in[0,U/2)$, $\tilde{f}_j\in[U/2,U],$ then $|f_j-(U-\tilde{f}_j)|\leqslant|f_j-\tilde{f}_j|.$ Therefore, one can get
		$|S^{k,i}(\bff)_{jj}-S^{k,i}(\tilde{\bff})_{jj}| \leqslant|\alpha^{k,i}_j/p_j|\cdot|f_j-\tilde{f}_j|$.
		Thus, $\|\Delta_2\|_2\leqslant \|\bfalpha^{k,i}\|_2/p_0\|\bff-\tilde{\bff}\|_2$, where $p_0:=\min\{p_j: j\in\bN_q\}$. Since $ \bfalpha^{k,i} $ and $ \nabla\Phi_i(\tilde{\bff}) $ are bounded and $p_0>0$, there exists a constant $c_3>0$ such that for all $k\in\bN_0, i\in\bN_M$,
		\begin{equation}\label{c_3}
		\|\nabla\Phi_i(\tilde{\bff})\| \|\Delta_2\|_2\leqslant c_3\|\bff-\tilde{\bff}\|_2.
		\end{equation}
		Let $ L\coloneqq c_2+c_3.$ It follows from \eqref{DELTA}, \eqref{c_2} and \eqref{c_3} that
		$\|\Delta\|_2\leqslant L \|\bff-\tilde{\bff}\|_2$.
		That is, $ \bfS^{k,i}(\bff)\nabla\Phi_i(\bff) $ are Lipschitz continuous with Lipschitz constants bounded above by $L$ for all $ k\in\bN_0, i\in\bN_M. $
	\end{proof}
	
	The proof of Lemma \ref{lem:monotone} is presented as follows.
	
	\begin{proof}
		Let $ \bfb^{k,i}\coloneqq \diag(\bfalpha-\bfdelta^{k,i})\nabla\Phi_i(\bff^{k,i-1})$ for $ k\in\bN_0, i\in\bN_M.$ For $f^{k,i-1}_j\in(0,U/2)$, $ i\in\bN_M$, we have that $ S^{k,i}(\bff^{k,i-1})_{{j}{j}}=(\alpha_j-\delta^{k,i}_{j}) f^{k,i}_{j}/p_{j}$. By assumption, $\bff^{k,i}=P_{t}(\tilde{\bff}^{k,i})=\tilde{\bff}^{k,i}.$ The definition of $\bff^{k,i}$ yields 
		$$f^{k,i}_{j}=f^{k,i-1}_{j}(1-\lambda_k/p_j b^{k,i}_j),$$
		which implies 
		$$f^{k+1}_{j}=f^k_{j}\Pi_{i=1}^M(1-\lambda_k/p_j b^{k,i}_j).$$ By the boundedness of $ b^{k,i}_j/p_j, $ we have 
		$$ f^{k+1}_j=f^k_j(1-\lambda_k/p_j\sum_{i=1}^{M}b^{k,i}_j+\cO(\lambda_k^2)). $$ We next estimate $\sum_{i=1}^{M} \bfb^{k,i}.$ By definition, $\bfb^{k,i}=\diag(\bfalpha)\nabla\Phi_i(\bff^{k,i-1})-\diag(\bfdelta^{k,i})\nabla\Phi_i(\bff^{k,i-1}).$ To estimate the first term on the right hand side of the last equation, we write $\nabla\Phi_i(\bff^{k,i-1})=\Delta^{k,i}+\nabla\Phi_i(\bff^k)$, where $\Delta^{k,i}:=\nabla\Phi_i(\bff^{k,i-1})-\nabla\Phi_i(\bff^k)$. It follows that $ \sum_{i=1}^{M}\nabla\Phi_i(\bff^{k,i-1})=\sum_{i=1}^{M}\Delta^{k,i}+\nabla\Phi(\bff^k).$
		By condition (iii), there exists $c_4>0$ such that   
		$
		\|\Delta^{k,i}\|_2\leqslant c_4\|\bff^{k,i-1}-\bff^k\|_2.
		$
		Since $ \tilde{\bff}^{k,i}\in\mathrm{int\,}\cB$, by \eqref{eq:IterationInterior} and Lemma \ref{lem:LipConPrecond}, there exists a constant $c_5>0$ such that
		\begin{equation}\label{eq:e1}
		\|\bff^{k,i-1}-\bff^k\|_2 \leqslant c_5 M \lambda_k.
		\end{equation}
		Hence  $\|\Delta^{k,i}\|_2 =\cO(\lambda_k)$ and this gives that $ \sum_{i=1}^{M}\nabla\Phi_i(\bff^{k,i-1})=\nabla\Phi(\bff^k)+\cO(\lambda_k). $ Moreover, by condition (iii), we have that $ \|\sum_{i=1}^{M}\diag(\bfdelta^{k,i})\nabla\Phi_i(\bff^{k,i-1})\|_2=\cO(\delta_k). $ Therefore, $$\sum_{i=1}^{M}\bfb^{k,i}=\diag(\bfalpha)\nabla\Phi(\bff^k)+\cO(\lambda_k)+\cO(\delta_k). $$
		Thus, we obtain  \eqref{eq:increasing}.  
		
		Equation \eqref{eq:decreasing} may be shown in a similar manner. Indeed, for $f^{k,i-1}_j\in[U/2,U)$,  $ i\in\bN_M $, we have that $ S^{k,i}(\bff^{k,i-1})_{{j}{j}}=\alpha^{k,i}_{j} (U-f^{k,i}_{j})/p_{j}$. The definition of $ \bff^{k,i} $ yields that 
		$$ (U-f^{k+1}_{j})=(U-f^k_{j})\Pi_{i=1}^M(1+\lambda_k/p_j b^{k,i}_j). $$ This equation with similar arguments leads to  \eqref{eq:decreasing}.
	\end{proof}
	
	The proof of Lemma \ref{lem:convergence} is presented as follows.
	
	\begin{proof} We first prove part (a). Let $ H(\bff) $ denote the Hessian matrix of $ \Phi(\bff) $, `$ \odot $' denote the component-wise multiplication of two vectors, and ${\bfh}^k:=\bff^{k+1}-\bff^k$. By the Taylor expansion, we have that
			\begin{equation}\label{eq:Taylor}
				\begin{aligned}
				\Phi(\bff^{k+1})=\Phi(\bff^k)+(\nabla\Phi(\bff^k))^\top{\bfh}^k+R_k,
				\end{aligned}
			\end{equation}
		where $R_k:= ({\bfh}^k)^\top H(\bff^k+\bftheta\odot{\bfh}^k){\bfh}^k$, for some vector $ \bftheta\in\bR^q_+$ with $ \theta_j\in(0,1) $ for all $j\in \bN_q$.  We now estimate $ R_k. $ By condition (iii), $ \nabla\Phi(\bff) $  is Lipschitz continuous. Hence $H(\bff)$ is bounded on $\cB.$ Then we have $ |R_k| = \cO (\|{\bfh}^k\|_2^2)$. We next evaluate the term ${\bfh}^k$. For notation simplicity, we let
			$
			\bfe^{k,i}\!\coloneqq\! \bfS^{k,i}(\bff^k)\nabla\Phi_i(\bff^k)\!-\!\bfS^{k,i}(\bff^{k,i-1})\nabla\Phi_i(\bff^{k,i-1}),
			$
			and $\bfe^k:=\sum_{i=1}^M \bfe^{k,i}.$
			By Lemma \ref{lem:LipConPrecond}, we have  
			$
			\|\bfe^{k,i}\|_2\leqslant L\|\bff^{k,i-1}-\bff^k\|_2.
			$
			This combined with \eqref{eq:e1} implies that $ \|\bfe^{k,i}\|_2=\cO(\lambda_k), $ and thus
			$
			\|\bfe^k\|_2=\cO(\lambda_k).
			$
			By assumption, $ \tilde{\bff}^{k,i}\in\mathrm{int\,}\cB$, from \eqref{eq:IterationInterior}, we obtain that
			\begin{equation}\label{eq:ite_Lip}
			{\bfh}^k=-\lambda_k\sum_{i=1}^M\bfS^{k,i}(\bff^k)\nabla\Phi_i(\bff^k)+\lambda_k\bfe^k.
			\end{equation}
			Let  $\bfd^k\coloneqq \bfS(\bff^k)\sum_{i=1}^{M}\diag(\bfdelta^{k,i})\nabla\Phi_i(\bff^k)$, and $ J_k:= (\nabla\Phi(\bff^k))^\top \diag(\bfalpha)\bfS(\bff^k)\nabla\Phi(\bff^k)$. 
			Then $\|\bfd^k\|_2=\cO(\delta_k)$ by condition (v) and Lemma \ref{lem:LipConPrecond}, and $J_k\geqslant 0 $ by the positive semi-definiteness of $ \diag(\bfalpha)\bfS(\bff^k)$. Since $ \bfS^{k,i}(\bff^k)=\diag(\bfalpha-\bfdelta^{k,i})\bfS(\bff^k),$ then we have
			\begin{equation}
			\begin{aligned}
			\sum_{i=1}^M\!\bfS^{k,i}\!(\bff^k)\nabla\Phi_i(\bff^k)\!=\!\diag(\bfalpha)\bfS(\bff^k)\nabla\Phi(\bff^k)\!-\!\bfd^k.
			\end{aligned}
			\end{equation}
			Combining this with \eqref{eq:ite_Lip}, we have
			\begin{equation}\label{eq:recursive_iter}
			{\bfh}^k=-\lambda_k\diag(\bfalpha)\bfS(\bff^k)\nabla\Phi(\bff^k)+\lambda_k(\bfd^k+\bfe^k).
			\end{equation}
		By the boundedness of $\diag(\bfalpha)\bfS(\bff)\nabla\Phi(\bff)$ and the norm of $ \bfe^k $ and $ \bfd^k $, we have $\|\bfh^k\|_2=\cO(\lambda_k)$ and hence, $|R_k|=\cO(\lambda_k^2).$
		This combined with \eqref{eq:Taylor}, \eqref{eq:recursive_iter} and the boundedness of $\nabla\Phi(\bff)$ yields that
		\begin{equation}\label{eq:Taylorfinal}
		\Phi(\bff^{k+1})=\Phi(\bff^k)-\lambda_k J_k+\cO(\lambda_k\delta_k)+\cO(\lambda_k^2).
		\end{equation}
		For $ s\in\bN, $ summing both sides of \eqref{eq:Taylorfinal} for $k$ from $0$ to $s$, we obtain that
		\begin{equation}\label{eq:Cauchy}
	    \Phi(\bff^{s+1})=\Phi(\bff^{0})+\!\sum_{k=0}^{s}[-\lambda_k J_k+\cO(\lambda_k\delta_k)+\cO(\lambda_k^2)].
		\end{equation}
		We now prove the convergence of the right hand side of \eqref{eq:Cauchy}. By condition (vi), $\sum_{k=0}^\infty \lambda_k\bfdelta^{k,i}$ is convergent, and hence $\sum_{k=0}^\infty \lambda_k\delta_k$ is convergent. 
		Notice the facts we have in hand: (i) $\sum_{k=0}^\infty\lambda_k^2<\infty$ (by condition (iv)); (ii) the convergence of $ \sum_{k=0}^\infty \lambda_k\delta_k,$ 
		it remains to show that  $\sum_{k=0}^\infty\lambda_k J_k$ is convergent. In view of the facts (i), (ii) and the boundedness of $\Phi(\bff)$ on $\cB$, the partial sum $\sum_{k=0}^s\lambda_k J_k$ is bounded, which combined with its monotonicity  ($ \lambda_kJ_k\geqslant0 $) implies its convergence.
		
		We next prove part (b). Since for each $k$, $J_k\geq 0$, there exists a subsequence $J^{k_n}$ such that $\lim_{n\to\infty} J^{k_n}=0$. In fact, assume to the contrary that there exists $\epsilon_0>0$ and $K_0\in\bN_+$ such that $ J_k\geqslant\epsilon_0$, for all $k>K_0.$ Because $\sum_{k=0}^\infty\lambda_k=\infty$, by condition (iv), and $ \lambda_k>0,$ we would have $\sum_{k=0}^s\lambda_k J_k\geqslant\epsilon_0\sum_{k=0}^s\lambda_k\to\infty,$ as $s\to \infty$, a contradiction. Moreover, since $ \bff^{k_n} $ is bounded, there exists a convergent subsequence $ \bff^{k'_n}$ having the limit $\bff^*\in\cB$. Thus, $(\nabla\Phi(\bff^*))^\top \diag(\bfalpha)\bfS(\bff^*)\nabla\Phi(\bff^*)=\lim_{n\to\infty} J^{k'_{n}}=0.$ Letting $r_j:= (\partial/\partial f_j)\Phi(\bff^*)$ and $s_j:=S(\bff^*)_{jj}$, the last equation yields $\sum_{j=1}^{q}\alpha_j s_jr_j^2=0.$ Since $s_j\geqslant 0$ and $\alpha_j>0$, we have  for all $j\in\bN_q$ that $\alpha_j s_j r_j^2\geqslant 0$, which implies that $ s_j r_j=0$ for all $j\in\bN_q $. That is, $ \bfS(\bff^*)\nabla\Phi(\bff^*)={\bf 0}.$ 

		Finally, we show part (c). 
		According to \cite[Page 203]{polyak1987introduction}, it suffices to prove that for each $j\in\bN_q$,  (i) if $0<f^*_j<U$, then $r_j=0$; (ii) if $f^*_j=0$, then $r_j\geqslant0$; and  if $f^*_j=U$, then $r_j\leqslant0$. 
		Case (i) clearly follows from part (b), from which we have 
		$s_j r_j=0$ for all $ j\in\bN_q.$ By the definition \eqref{eq:preconditioner} of $\bfS(\bff)$, $s_j>0$ for $0<f^*_j<U$, and thus, $r_j=0$. 		
		
		It remains to prove case (ii). To this end, we let $\cJ_1\coloneqq\left\lbrace j'\in\bN_q: f^*_{j'}=0, r_{j'}<0\right\rbrace,$ $ \cJ_2\!\coloneqq\!\left\lbrace j''\in\bN_q\!:\! f^*_{j''}=U, r_{j''}\!>0\right\rbrace $, $ \cJ\coloneqq\cJ_1\!\cup\!\cJ_2$, and show $\cJ=\emptyset.$ Assume to the contrary that  $\cJ\neq\emptyset$. Then,  either $\cJ_1\neq\emptyset$ or $\cJ_2\neq\emptyset.$ Assume $\cJ_1\neq\emptyset,$ then for any $j'\in\cJ_1,$ since $\nabla\Phi(\bff)$ is continuous at $\bff^*$, there exists $ \delta\in(0,U/4) $ such that for all $\bff\in\cB_{\delta}:= \left\lbrace\bff\in\cB:\|\bff-\bff^*\|_2<\delta\right\rbrace$, there holds $(\partial/\partial f_{j'})\Phi(\bff)<0$. By Lemma \ref{lem:difference}, there exists $ K_1\in\bN $ such that for $ k>K_1, $  $ \|\bff^{k,i-1}-\bff^k\|_2<\delta. $ Then for $ k>K_1, $ if $ \bff^k\in\cB_{\delta}, $ we have $ \|\bff^{k,i-1}-\bff^*\|_2<2\delta<U/2, $ and hence, $ \bff^{k,i-1}_{j'}\in(0,U/2) $ for all $ i\in\bN_M, j'\in\cJ_1$. In this case, $ S^{k,i}(\bff^{k,i-1})_{{j'}{j'}}=\alpha^{k,i}_{j'} f^{k,i}_{j'}/p_{j'}$ for all $ i\in\bN_M,$ and hence Lemma \ref{lem:monotone} ensures that \eqref{eq:increasing} holds. Since $ (\partial/\partial f_{j'})\Phi(\bff^k)<0 $ for $ \bff^k\in\cB_\delta$ and $\lambda_k,\delta_k\to0$ as $ k\to\infty,$ then there exists $ K_2>K_1 $ such that if $ k>K_2 $ and $ \bff^k\in\cB_{\delta}, $ we have $ f^{k+1}_{j'}>f^k_{j'} $ for any $ {j'}\in\cJ_1. $   Therefore,  for $ k>K_2$, if $ \bff^k\in\cB_{\delta}$,  then we have $ f^{k+1}_{j'}>f^{k}_{j'} $ for any $ {j'}\in\cJ_1. $ Similarly, if $\cJ_2\neq\emptyset,$ for any $j''\in\cJ_2,$ there exists $ K_3>K_2 $ such that for $ k>K_3$, if $ \bff^k\in\cB_\delta $, then \eqref{eq:decreasing} ensures that $ f^{k+1}_{j''}>f^k_{j''} $.\par
		
		Since $\lim_{n\to\infty}\bff^{k'_n}=\bff^*,$ there exists $ K_4>K_3 $ such that if $ k'_n>K_4$, $ \bff^{k'_n}\in\cB_{\delta}. $  
		Suppose $ k'_{n_0}>K_4 $ for some $ n_0. $ Let $t_n:=\max\left\lbrace k: K_4< k<k'_{n+n_0}, \bff^k\notin\cB_{\delta}\right\rbrace.$ If  for some $n,$ $\bff^k\in\cB_{\delta}$ for all $K_4<k<k'_{n+n_0}$, set $t_n:=K_4,$ and hence $ t_n\geqslant K_4. $ Clearly, we have $\bff^k\in\cB_{\delta}$ if $t_n+1\leqslant k\leqslant k'_{n+n_0}$. Moreover, $t_n$ is a monotone increasing sequence. Then, either (a) $\lim_{n\to\infty}t_n=t_0\geqslant K_4,$ or (b)  $\lim_{n\to\infty}t_n=+\infty.$ 
		If it is the case (a), then $\bff^k\in\cB_{\delta}$ for all $k>t_0$. Thus, for $m>l>t_0$ that $f^m_{j'}>f^l_{j'}>0$ for $j'\in\cJ_1$. This contradicts the fact that $f^*_{j'}=0$ for $j'\in\cJ_1$. Hence, it must be the case (b). Since $\bff^k\in\cB_{\delta}$ for $t_n+1\leqslant k\leqslant k'_{n+n_0}$, we have that $f^{k'_{n+n_0}}_{j'}\geqslant f^{t_n+1}_{j'}>0$ for $j'\in\cJ_1$ and $f^{k'_{n+n_0}}_{j''}\leqslant f^{t_n+1}_{j''}<U$ for $ j''\in\cJ_2 $. It follows that $\lim_{n\to\infty}f^{t_n+1}_{j'}=0$ for $j'\in\cJ_1$ and $\lim_{n\to\infty}f^{t_n+1}_{j''}=U$ for $j''\in\cJ_2$. Then, Lemma \ref{lem:difference} ensures that $\lim_{n\to\infty}f^{t_n}_{j'}=0$ for $j'\in\cJ_1$ and $\lim_{n\to\infty}f^{t_n}_{j''}=U$ for $j''\in\cJ_2.$ Thus, we can find a convergent subsequence $ \bff^{t_{n_l}}$ of $ \bff^{t_n}$ such that $\lim_{l\to\infty}\bff^{t_{n_l}}=\bff^{**}$ with $f^{**}_{j'}=0$ for $j'\in\cJ_1 $ and $f^{**}_{j''}=U$ for $j''\in\cJ_2. $ Since $\bff^{t_n}\notin\cB_{\delta}$, we observe that $\bff^{**}\notin\cB_{\delta}$, which ensures that $\bff^{**}\neq\bff^*$. By part (a), $\Phi(\bff^k)$ is convergent, which implies that $\Phi(\bff^{**})=\Phi(\bff^*).$ Let $\cD\coloneqq\left\lbrace\bff\in\cB: f_{j'}=0 \text{~for~} j'\in\cJ_1 \text{ and } f_{j''}=U \text{ for } j''\in\cJ_2\right\rbrace$. Thus, $\bff^*,\bff^{**}\in\cD.$ It can be verified for any $\bff\in\cD$ that $\left\langle\bff-\bff^*,\nabla\Phi(\bff^*)\right\rangle\geqslant0$. By \cite[Page 203]{polyak1987introduction}, $\bff^*$ is a minimizer of $\Phi$ over $\cD$. Hence, $\Phi$ has two different minimizers $\bff^*$ and $\bff^{**}$ over the convex set $\cD.$ This contradicts the assumption that $\Phi$ has a unique minimizer on $\cB$. Thus, we have that $\cJ=\emptyset.$
	\end{proof}
	Here is the proof of Theorem \ref{thm:SpecificSDPs}.
	
	\begin{proof}
	    From Proposition \ref{prop:StrictlyConvex} and Proposition \ref{prop:LipGrad}, we have that $ \Phi $ satisfies conditions (i)-(iii). One can directly obtain that the relaxation sequence $ \lambda_k $ satisfies condition (iv).\par
		By Theorem \ref{thm:convergenceSDP}, to prove the convergence of SDP-BSREM algorithm with $ \lambda_k $ and $\bfS^{k,i}$, it is sufficient to show that $\lambda_k$ and $\bfS^{k,i}$ satisfy conditions (v) and (vi). To do this, for the subiteration-dependent preconditioner $\bfS^{k,i}(\bff)=\diag(\alpha_{k,i}\bfnu^{k,i})\bfS(\bff)$, one need to show that $ \lim_{k\to\infty}\alpha_{k,i}=\alpha>0$ for all $ i\in\bN_M $, and $ \sum_{k=0}^{\infty}\lambda_k(\alpha-\alpha_{k,i}) $ converges for all $ i\in\bN_M$ since $\bfnu^{k,i}$ is a positive vector sequence and
		$\bfnu^{k,i}=\bfnu^{k_1,M}$ for $k> k_1$ and $i\in\bN_M.$\par
		For $ \alpha_{k,i} $ defined in \eqref{preconditioner:Nesterov}, the sequence $ t_{k,i} $ is increasing. By induction, we have that $ t_{k,i}>(kM+i)/2. $ Further, it can be shown that $ \lim_{k\to\infty}t_{k,i}/k=M/2 $, and thus $ \lim_{k\to\infty}t_{k,i}/t_{k,i+1}=1 $ for all $ i\in\bN_M. $ Then we can obtain $ \lim_{k\to\infty}\alpha_{k,i}=1+\lim_{k\to\infty}(t_{k,i}-1)/t_{k,i+1}=2>0. $ By computing
		$
		2-\alpha_{k,i}=1/(2t_{k,i+1}(\sqrt{1+4t^2_{k,i}}+4t_{k,i}))+3/(2t_{k,i+1}),
		$
		we have that $ \lim_{k\to\infty}k(2-\alpha_{k,i})=3/M.$ Thus the series $ \sum_{k=0}^{\infty}\lambda_k(2-\alpha_{k,i}) $ converges since $ \sum_{k=0}^{\infty}1/k^2<\infty $. Therefore, for preconditioners P1 or M1 and relaxation $ \lambda_k, $ conditions (v) and (vi) are satisfied. \par
		For $\alpha_{k,i}$ defined in \eqref{preconditioner:KM}, we have that $ \alpha_{k,i} $ is monotone and $ \lim_{k\to\infty}\alpha_{k,i}=\varrho>0. $ It can be shown that $ \lim_{k\to\infty}k(\varrho-\alpha_{k,i})=(\varrho\delta_1-\delta_2)/M. $ Hence for precondition P2 or M2 and relaxation $ \lambda_k $, conditions (v) and (vi) are satisfied.
	\end{proof}
	
	\bibliographystyle{IEEEtran}
	
	\bibliography{SDP-BSREM_arxiv_20220608.bib}

\begin{thebibliography}{10}
\providecommand{\url}[1]{#1}
\csname url@samestyle\endcsname
\providecommand{\newblock}{\relax}
\providecommand{\bibinfo}[2]{#2}
\providecommand{\BIBentrySTDinterwordspacing}{\spaceskip=0pt\relax}
\providecommand{\BIBentryALTinterwordstretchfactor}{4}
\providecommand{\BIBentryALTinterwordspacing}{\spaceskip=\fontdimen2\font plus
\BIBentryALTinterwordstretchfactor\fontdimen3\font minus
  \fontdimen4\font\relax}
\providecommand{\BIBforeignlanguage}[2]{{%
\expandafter\ifx\csname l@#1\endcsname\relax
\typeout{** WARNING: IEEEtran.bst: No hyphenation pattern has been}%
\typeout{** loaded for the language `#1'. Using the pattern for}%
\typeout{** the default language instead.}%
\else
\language=\csname l@#1\endcsname
\fi
#2}}
\providecommand{\BIBdecl}{\relax}
\BIBdecl

\bibitem{Shepp1982maximum}
L.~A. Shepp and Y.~Vardi, ``Maximum likelihood reconstruction for emission
  tomography,'' \emph{IEEE Trans. Med. Imag.}, vol.~1, no.~2, pp. 113--122,
  Oct. 1982.

\bibitem{Lange1984em}
K.~Lange and R.~Carson, ``{EM} reconstruction algorithms for emission and
  transmission tomography,'' \emph{J. Comput. Assist. Tomogr.}, vol.~8, no.~2,
  pp. 306--316, 1984.

\bibitem{ahn2015quantitative}
S.~Ahn, S.~G. Ross, E.~Asma, J.~Miao, X.~Jin, S.~D.~W. Lishui~Cheng, and R.~M.
  Manjeshwar1, ``Quantitative comparison of {OSEM} and penalized likelihood
  image reconstruction using relative difference penalties for clinical
  {PET},'' \emph{Phys. Med. Biol.}, vol.~60, no.~15, pp. 5733--5751, 2015.

\bibitem{Hudson1994Accelerated}
H.~M. {Hudson} and R.~S. {Larkin}, ``Accelerated image reconstruction using
  ordered subsets of projection data,'' \emph{IEEE Trans. Med. Imag.}, vol.~13,
  no.~4, pp. 601--609, Dec. 1994.

\bibitem{Browne1996alternative}
J.~Browne and A.~R. De~Pierro, ``A row-action alternative to the {EM} algorithm
  for maximizing likelihood in emission tomography,'' \emph{IEEE Trans. Med.
  Imag.}, vol.~15, no.~5, pp. 687--699, Oct. 1996.

\bibitem{Bertsekas1997A}
D.~P. Bertsekas, ``A new class of incremental gradient methods for least
  squares problems,'' \emph{SIAM J. Optim.}, vol.~7, no.~4, pp. 913--926, 1997.

\bibitem{de2001fast}
A.~R. {De Pierro} and M.~E.~B. Yamagishi, ``Fast {EM}-like methods for maximum"
  a posteriori" estimates in emission tomography,'' \emph{IEEE Trans. Med.
  Imag.}, vol.~20, no.~4, pp. 280--288, Apr. 2001.

\bibitem{ahn2003globally}
S.~Ahn and J.~A. Fessler, ``Globally convergent image reconstruction for
  emission tomography using relaxed ordered subsets algorithms,'' \emph{IEEE
  Trans. Med. Imag.}, vol.~22, no.~5, pp. 613--626, May 2003.

\bibitem{tsai2018fast}
Y.-J. Tsai, A.~Bousse, M.~J. Ehrhardt, C.~W. Stearns, S.~Ahn, B.~F. Hutton,
  S.~Arridge, and K.~Thielemans, ``Fast quasi-{Newton} algorithms for penalized
  reconstruction in emission tomography and further improvements via
  preconditioning,'' \emph{IEEE Trans. Med. Imag.}, vol.~37, no.~4, pp.
  1000--1010, Apr. 2018.

\bibitem{rudin1992nonlinear}
L.~I. Rudin, S.~Osher, and E.~Fatemi, ``Nonlinear total variation based noise
  removal algorithms,'' \emph{Phys. D}, vol.~60, no. 1-4, pp. 259--268, 1992.

\bibitem{bredies2010total}
K.~Bredies, K.~Kunisch, and T.~Pock, ``Total generalized variation,''
  \emph{SIAM J. Imag. Sci.}, vol.~3, no.~3, pp. 492--526, 2010.

\bibitem{Geman1987Statistical}
S.~Geman and D.~E. McClure, ``Statistical methods for tomographic image
  reconstruction,'' \emph{Bull. Int. Statist. Inst.}, vol.~52, no.~4, pp.
  5--21, 1987.

\bibitem{Mumcuoglu1996Bayesian}
E.~{\"U}. Mumcuoglu, R.~M. Leahy, and S.~R. Cherry, ``Bayesian reconstruction
  of {PET} images: methodology and performance analysis,'' \emph{Phys. Med.
  Biol.}, vol.~41, no.~9, pp. 1777--1807, 1996.

\bibitem{Zhang2016investigation}
Z.~Zhang, J.~Ye, B.~Chen, A.~E. Perkins, S.~Rose, E.~Y. Sidky, C.-M. Kao,
  D.~Xia, C.-H. Tung, and X.~Pan, ``Investigation of optimization-based
  reconstruction with an image-total-variation constraint in {PET},''
  \emph{Phys. Med. Biol.}, vol.~61, no.~16, pp. 6055--6084, Aug. 2016.

\bibitem{chambolle2011first}
A.~Chambolle and T.~Pock, ``A first-order primal-dual algorithm for convex
  problems with applications to imaging,'' \emph{J. Math. Imag. Vis.}, vol.~40,
  no.~1, pp. 120--145, 2011.

\bibitem{krol2012preconditioned}
A.~Krol, S.~Li, L.~Shen, and Y.~Xu, ``Preconditioned alternating projection
  algorithms for maximum a posteriori {ECT} reconstruction,'' \emph{Inverse
  Problems}, vol.~28, no.~11, p. 115005, 2012.

\bibitem{li2015effective}
S.~Li, J.~Zhang, A.~Krol, C.~R. Schmidtlein, L.~Vogelsang, L.~Shen, E.~Lipson,
  D.~Feiglin, and Y.~Xu, ``Effective noise-suppressed and artifact-reduced
  reconstruction of {SPECT} data using a preconditioned alternating projection
  algorithm,'' \emph{Med. Phys.}, vol.~42, no.~8, pp. 4872--4887, 2015.

\bibitem{nuyts2002concave}
J.~Nuyts, D.~Bequ{\'e}, P.~Dupont, and L.~Mortelmans, ``A concave prior
  penalizing relative differences for maximum-a-posteriori reconstruction in
  emission tomography,'' \emph{IEEE Trans. Nucl. Sci.}, vol.~49, no.~1, pp.
  56--60, Feb. 2002.

\bibitem{wang2012penalized}
G.~Wang and J.~Qi, ``Penalized likelihood {PET} image reconstruction using
  patch-based edge-preserving regularization,'' \emph{IEEE Trans. Med. Imag.},
  vol.~31, no.~12, pp. 2194--2204, Dec. 2012.

\bibitem{wang2015edge}
------, ``Edge-preserving {PET} image reconstruction using trust optimization
  transfer,'' \emph{IEEE Trans. Med. Imag.}, vol.~34, no.~4, pp. 930--939, Apr.
  2015.

\bibitem{mumcuoglu1994fast}
E.~{\"U}. Mumcuoglu, R.~Leahy, S.~R. Cherry, and Z.~Zhou, ``Fast gradient-based
  methods for bayesian reconstruction of transmission and emission {PET}
  images,'' \emph{IEEE Trans. Med. Imag.}, vol.~13, no.~4, pp. 687--701, Dec.
  1994.

\bibitem{lin2019krasnoselskii}
Y.~Lin, C.~R. Schmidtlein, Q.~Li, S.~Li, and Y.~Xu, ``A krasnoselskii-mann
  algorithm with an improved {EM} preconditioner for {PET} image
  reconstruction,'' \emph{IEEE Trans. Med. Imag.}, vol.~38, no.~9, pp.
  2114--2126, Sep. 2019.

\bibitem{Ehrhardt2019fast}
M.~J. Ehrhardt, P.~Markiewicz, and C.-B. Sch{\"o}nlieb, ``Faster {PET}
  reconstruction with non-smooth priors by randomization and preconditioning,''
  \emph{Phys. Med. Biol.}, vol.~64, no.~22, p. 225019, Nov. 2019.

\bibitem{yu2002edge}
D.~F. Yu and J.~A. Fessler, ``Edge-preserving tomographic reconstruction with
  nonlocal regularization,'' \emph{IEEE Trans. Med. Imag.}, vol.~21, no.~2, pp.
  159--173, Feb. 2002.

\bibitem{Nesterov1983Method}
Y.~E. Nesterov, ``A method of solving a convex programming problem with
  convergence rate {$\mathit{O}$}($1/k^2$),'' \emph{Sov. Math. Dokl.}, vol.~27,
  no.~2, pp. 372--376, 1983.

\bibitem{beck2009fast}
A.~Beck and M.~Teboulle, ``A fast iterative shrinkage-thresholding algorithm
  for linear inverse problems,'' \emph{SIAM J. Imag. Sci.}, vol.~2, no.~1, pp.
  183--202, 2009.

\bibitem{Zeng2018FixFISTA}
X.~Zeng, L.~Shen, and Y.~Xu, ``A convergent fixed-point proximity algorithm
  accelerated by {FISTA} for the l0 sparse recovery problem,'' in
  \emph{Imaging, Vision and Learning Based on Optimization and PDEs}, X.-C.
  Tai, E.~Bae, and M.~Lysaker, Eds.\hskip 1em plus 0.5em minus 0.4em\relax
  Cham, Switzerland: Springer Int. Publishing, 2018, pp. 27--45.

\bibitem{kim2015combining}
D.~Kim, S.~Ramani, and J.~A. Fessler, ``Combining ordered subsets and momentum
  for accelerated {X}-ray {CT} image reconstruction,'' \emph{IEEE Trans. Med.
  Imag.}, vol.~34, no.~1, pp. 167--178, Jan. 2015.

\bibitem{Li2006A}
Q.~Li, E.~Asma, S.~Ahn, and R.~M. Leahy, ``A fast fully {4-D} incremental
  gradient reconstruction algorithm for list mode {PET} data,'' \emph{IEEE
  Trans. Med. Imag.}, vol.~26, no.~1, pp. 58--67, Jan. 2007.

\bibitem{schmidtlein2017relaxed}
C.~R. Schmidtlein, Y.~Lin, S.~Li, A.~Krol, B.~J. Beattie, J.~L. Humm, and
  Y.~Xu, ``Relaxed ordered subset preconditioned alternating projection
  algorithm for {PET} reconstruction with automated penalty weight selection,''
  \emph{Med. Phys.}, vol.~44, no.~8, pp. 4083--4097, 2017.

\bibitem{moses2011fundamental}
W.~W. Moses, ``Fundamental limits of spatial resolution in {PET},'' \emph{Nucl.
  Instrum. Methods Phys. Res. A, Accel. Spectrom. Detect. Assoc. Equip.}, vol.
  648, pp. S236--S240, Aug. 2011.

\bibitem{berthon2015petstep}
B.~Berthon, I.~H{\"a}ggstr{\"o}m, A.~Apte, B.~J. Beattie, A.~S. Kirov, J.~L.
  Humm, C.~Marshall, E.~Spezi, A.~Larsson, and C.~R. Schmidtlein, ``{PETSTEP}:
  generation of synthetic {PET} lesions for fast evaluation of segmentation
  methods,'' \emph{Phys. Med.}, vol.~31, no.~8, pp. 969--980, 2015.

\bibitem{GEPETToolbox5.0}
S.~Ross and K.~Thielemans, \emph{General Electric PET Toolbox Release 5.0},
  2011--2019.

\bibitem{Acr2012pet}
A.~A. of~Physicists~in Medicine \emph{et~al.}, \emph{{PET} phantom instructions
  for evaluation of {PET} image quality}, 2012.

\bibitem{macfarlane2006acr}
C.~R. MacFarlane, ``Acr accreditation of nuclear medicine and {PET} imaging
  departments,'' \emph{J. Nucl. Med. Technol.}, vol.~34, no.~1, pp. 18--24,
  Mar. 2006.

\bibitem{kim2019time}
K.~Kim, D.~Kim, J.~Yang, G.~El~Fakhri, Y.~Seo, J.~A. Fessler, and Q.~Li, ``Time
  of flight pet reconstruction using nonuniform update for regional recovery
  uniformity,'' \emph{Med. Phys}, vol.~46, no.~2, pp. 649--664, 2019.

\bibitem{Nesterov2004Introductory}
Y.~E. Nesterov, \emph{Introductory Lectures on Convex Optimization: A Basic
  Course}.\hskip 1em plus 0.5em minus 0.4em\relax New York: Kluwer, 2004.

\bibitem{polyak1987introduction}
B.~T. Polyak, \emph{Introduction to optimization}.\hskip 1em plus 0.5em minus
  0.4em\relax New York, NY, USA: Optim. Softw., 1987.

\end{thebibliography}

}

\end{document}